\documentclass[12pt]{amsart}
\usepackage{amsfonts,latexsym,rawfonts,amsmath,amssymb,amsthm,graphicx,color}
\usepackage{mathrsfs}
\usepackage{enumitem}
\usepackage{comment}
\usepackage{url}
\numberwithin{equation}{section}
\newtheorem{theorem}{Theorem}[section]
\newtheorem{lem}[theorem]{Lemma}
\newtheorem{thm}[theorem]{Theorem}
\newtheorem{pro}[theorem]{Proposition}
\newtheorem{cor}[theorem]{Corollary}
\newtheorem{defi}[theorem]{Definition}

\newtheorem{exa}[theorem]{Example}

\def\s{\,\,\,\,}
\def\lan{\langle}
\def\ran{\rangle}
\def\dis{\displaystyle}

\def\mv{1.7ex}

\def\R{\mathbb{R}}

\def\vint{\mathop{\mathchoice%
		{\setbox0\hbox{$\displaystyle\intop$}\kern 0.22\wd0%
			\vcenter{\hrule width 0.6\wd0}\kern -0.82\wd0}%
		{\setbox0\hbox{$\textstyle\intop$}\kern 0.2\wd0%
			\vcenter{\hrule width 0.6\wd0}\kern -0.8\wd0}%
		{\setbox0\hbox{$\scriptstyle\intop$}\kern 0.2\wd0%
			\vcenter{\hrule width 0.6\wd0}\kern -0.8\wd0}%
		{\setbox0\hbox{$\scriptscriptstyle\intop$}\kern 0.2\wd0%
			\vcenter{\hrule width 0.6\wd0}\kern -0.8\wd0}}%
	\mathopen{}\int}
\allowdisplaybreaks[4]

\title[3-circle inequalities]{\bf Three-Circle Inequalities for  Almost Harmonic Functions on Cylinders}
\author[Y. Li \& W. Qi]
{Yuxiang Li, Wenjuan Qi}

\address[Yuxiang Li]{ Department of Mathematical Sciences, Tsinghua University, Beijing 100084, P.R. China}
	\email{liyuxiang@mail.tsinghua.edu.cn}

\address[Wenjuan Qi]{ Department of Mathematical Sciences, Tsinghua University, Beijing 100084, P.R. China}
	\email{qiwenjuan2026@163.com}
\date{}

\begin{document}
	
\begin{abstract}
   The three-circle method was introduced into the blow-up analysis of harmonic maps by Qing and Tian [Comm. Pure Appl. Math. 1997]. We extend this method to a class of \((\tau,\epsilon)\)-almost harmonic functions on long cylinders and establish three-circle inequalities for their weighted energies. These inequalities yield a dichotomy between exponential energy decay and asymptotic dominance of a critical Fourier mode, together with a characterization of the corresponding normalized limiting profiles. In the unweighted case, we prove convergence of the averaged neck curves and identify their limiting equations under additional structural assumptions. Applications include neck analysis and energy identities for harmonic maps, solutions of Rivi\`ere's equation, \(\alpha\)-harmonic maps, biharmonic approximations of harmonic maps, and Willmore surfaces.
\end{abstract}

\maketitle

	{\small
		\smallskip
		\newline
	}
	\bigskip

\section{Introduction}
The three-circle lemma refers to a quantitative inequality of the form
\begin{equation}\label{eqn:qphi-intro}
\Phi_2 \leq e^{-qL}(\Phi_1 + \Phi_3),
\end{equation}
for a nonnegative sequence $\{\Phi_i\}_{i=1}^3$, 
in which case we say that $\{\Phi_i\}$ satisfies the three-circle lemma for $(q,L)$.
The usefulness of \eqref{eqn:qphi-intro} is that we can show for each $q'\in(0, q)$, 
 the existence of $L_1(q, q')$ such that if $L > L_1$, then 
$$
\text{either}\,\, \Phi_{i} \geq e^{-q'L} \Phi_{i-1}  \quad \text{or} \quad \Phi_{i} \geq e^{-q' L} \Phi_{i+1}
$$
and,
moreover
$$
\Phi_{i-1} \leq e^{-q' L} \Phi_{i}\ \Longrightarrow\ \Phi_{i} \leq e^{-q' L} \Phi_{i+1},
\qquad
\Phi_{i+1} \leq e^{-q' L} \Phi_{i}\ \Longrightarrow\ \Phi_{i} \leq e^{-q' L} \Phi_{i-1}.
$$
As a consequence, if a nonnegative sequence $\{\Phi_i\}_{i=1}^m$ 
satisfies the three-circle lemma for $(q,L)$ on every three consecutive indices, 
then there exists a constant $C>0$ such that
\[
\Phi_i \leq C\big( e^{-qi} + e^{-q(m-i)} \big),
\qquad 1 \leq i \leq m.
\]
In particular, it plays a key role in the study of removable 
singularities, gap phenomena, and provides a powerful tool for the analysis of neck regions in blow-up problems.

The terminology goes back to the classical three-circle theorem of Hadamard on the logarithmic convexity of the maximum modulus of holomorphic functions. In geometric blow-up analysis, the method was introduced by Qing and Tian \cite{QingTian1997} in their  study of the bubbling phenomenon for almost harmonic maps from surfaces. By expanding the maps in Fourier series along long cylinders, they established a three-circle inequality for the $L^2$-energies of the maps on three consecutive subcylinders, which forces the energy to decay exponentially in at least one direction. The method was developed further by Chen and Tian \cite{ChenTian1999}, who proved that harmonic maps---and, more generally, $\alpha$-harmonic maps---on long cylinders converge exponentially to closed geodesics of the target manifold, a fact that lies at the heart of the compactness theory for moduli spaces of harmonic maps.

However, unfortunately, this line of ideas was not systematically developed further in the study of harmonic maps after these works. One reason is that, in many situations, the blow-up analysis for harmonic maps can be effectively handled by alternative methods introduced by Sacks and Uhlenbeck, based on integral identities and associated ordinary differential inequalities (cf \cite{SacksUhlenbeck1981}, \cite{DingTian1995}, \cite{LiWang2010}). As a consequence, the three-circle lemma, while powerful, has not become a standard tool in the harmonic map literature.

More recently, the 3-circle lemma was adapted  to the study of Willmore surfaces \cite{LiYin2025,LiYinZhou2024}, 
where it provides an effective tool for analyzing neck regions.  The application of the three-circle lemma to Willmore surfaces 
suggests that these ideas admit a broader formulation. This 
naturally leads us to develop a unified framework that encompasses 
both harmonic maps and Willmore surfaces, and extends to more 
general two-dimensional elliptic systems that are close to being 
harmonic setting.

A standard strategy in neck analysis is to first control the energy in the angular direction and then use a Pohozaev identity to recover information on the radial energy. When such a
Pohozaev estimate is not available, the radial energy cannot be
controlled directly by the angular energy. In this case, one needs a
more careful estimate of the radial energy on long thin cylinders.
A direct ODE argument  can become quite
complicated. By contrast, the three-circle method provides a much simpler and more robust way to capture the propagation  of the radial energy along the neck. 

For this aim,  we introduce a class of functions 
that are close to being harmonic in a scale-dependent sense.

\begin{defi}\label{defi of almost hm}
Let $\tau\colon(0,r_0)\to \mathbb{R}^+$ be a positive function. 
We say that $u\in W^{2,p}(Q=Q_1\cup Q_2\cup Q_3,\R^n)$, $p>1$, is a 
$(\tau,\epsilon)$-almost harmonic function if for any  
$r<r_0$, the following estimates hold:
\begin{gather}\label{almost.harmonic.1}
\|\nabla^2 u\|_{L^p(S^1\times[r,3L-r])} 
\leq \tau(r)\,\|\nabla u\|_{L^2(Q)}, \\ \label{almost.harmonic.2}
\|\Delta u\|_{L^1(S^1\times[r,3L-r])} \leq \tau(r)\,\epsilon\,\|\nabla u\|_{L^2(Q)}.
\end{gather}
\end{defi}


Functions of this type arise naturally in the analysis of neck regions in blow-up problems.
We give several examples.
\medskip

\noindent\textbf{Example 1.} Solutions $u \in W^{1,2}(D, \mathbb{R}^n)$ to Rivi\`ere's equation $$\Delta u = \Omega \cdot \nabla u,\quad\Omega \in L^2(D,so_n \otimes \mathbb{R}^2).$$ By $\epsilon$-regularity, $u$ is almost harmonic when $\|\Omega\|_{L^2}$ is sufficiently small (cf. \cite{Riviere2007}).

\medskip

\noindent\textbf{Example 2.} 
Conformal immersions $f\colon D\to \mathbb{R}^n$, which satisfy
\[
\Delta f = \frac{H}{2}|\nabla f|^2,
\]
where $H$ denotes the mean curvature. Such maps are almost harmonic 
when the Willmore energy is sufficiently small (cf. \cite{ChenLi2014}).

\medskip

\noindent\textbf{Example 3.} 
Harmonic maps $u\colon D\to M$, where $M$ is a smooth compact submanifold of $\mathbb{R}^n$. 
By $\epsilon$-regularity, $u$ is almost harmonic when 
$\|\nabla u\|_{L^2}$ is sufficiently small (cf. \cite{Schoen1984}).

\medskip

\noindent\textbf{Example 4.} 
$\alpha$-harmonic maps are almost harmonic when both $(\alpha-1)$ 
and $\|\nabla u\|_{L^2}$ are sufficiently small (cf. \cite{SacksUhlenbeck1981}).

\medskip

\noindent\textbf{Example 5.} 
Let $f$ be a conformal Willmore immersion from $D$ into $\mathbb{R}^n$, 
and let $H$ denote its mean curvature. Then $H$ is almost harmonic 
when $\|A\|_{L^2}$ is sufficiently small (cf.\ Section~4). To the best of our knowledge, 
this observation has not been explicitly noted in the literature.

\medskip


\noindent\textbf{Example 6.} 
Let $u\colon D\to\mathbb{R}^3$ be a conformal parametrization of a surface with prescribed mean curvature $H$. In isothermal coordinates, $u$ satisfies the $H$-surface equation \[
\Delta u=2H(u)\,(u_x\times u_y).
\]
Consequently, $u$ is almost harmonic when
$\|H(u)\nabla u\|_{L^2}$ is sufficiently small
(cf. \cite{Riviere2007}).

\medskip

\noindent\textbf{Example 7.} 
Let $u\colon D\to N\subset\mathbb{R}^{n}$ be a magnetic harmonic map, where $N$ is compact and equipped with a smooth closed $3$-form. By $\epsilon$-regularity, $u$ is almost harmonic when $\|\nabla u\|_{L^{2}(D)}$ is sufficiently small (cf. \cite{Branding2015}).

\medskip

\noindent\textbf{Example 8.} 
Let \(f\colon (M,\gamma)\to N\subset\mathbb R^n\) be affine harmonic, with \(M\) a K\"ahler affine surface and \(N\) a compact Riemannian submanifold. In affine coordinates, 
\[
\tau_\gamma(f)
=-\gamma^{\alpha\beta}
\Gamma^\delta_{\alpha\beta}(\gamma)\,\partial_\delta f
=:L\cdot\nabla f,
\]
Hence, \(f\) is almost harmonic when both \(\|\nabla f\|_{L^2}\) and \(\|L\|_{L^2}\) are sufficiently small (cf. \cite{JS09}).

\medskip

To incorporate energies arising in problems such as Willmore surfaces, 
it is natural to consider weighted energy functionals on cylinders. 
Given $Q_{a,b}=S^1\times [a,b]$ and $m\in \mathbb{Z}$, we define
\[
E(u,Q_{a,b})=\int_{Q_{a,b}} e^{-2mt}|\nabla u|^2.
\]
We write $u$ in Fourier series along the $S^1$-direction:
\[
u(t,\theta)=\frac{a_0(t)}{2}+\sum_{k=1}^\infty (a_k(t)\cos k\theta+b_k(t)\sin k\theta).
\]
We introduce
\[
\hat b_0=0,\qquad \hat a_0=a_0',\qquad 
\hat a_p=a_{|p|}'+p\,a_{|p|},\qquad 
\hat b_p=b_{|p|}'+p\,b_{|p|},
\]
for $p=\pm1,\pm2,\cdots$.
A direct computation shows that the weighted energy can be written as
\[
E(u,Q_{a,b})
=\frac{\pi}{2}\sum_{p=-\infty}^\infty \int_a^b e^{-2mt}
\big(|\hat a_p|^2+|\hat b_p|^2\big)\,dt.
\]
When $u$ is harmonic, we have
$$
\hat a_p=c_pe^{pt},\s\hat b_p=c_p'e^{pt}, \s p\neq 0,
$$
where $c_p$ and $c_p'$ are constants. From this formula it immediately follows that the critical Fourier mode $(\hat a_m,\hat b_m)$ is the obstruction to the three-circle inequality, since the corresponding summand $e^{2(p-m)t}$ is constant in $t$ when $p=m$.Obviously. 
When this mode vanishes, the three-circle lemma for $E(u,Q_i)$ holds, where $Q_i=S^1\times[(i-1)L,iL]$. 

For a general function $u$, the three-circle inequality holds 
provided that the coefficients $\hat a_m$ and $\hat b_m$ are 
negligible compared to the energy $E(u,Q_i)$. Motivated by this, we introduce the quantities
\[
\lambda_i
= \frac{e^{-mt_i}\big(\hat a_m(t_i)^2+\hat b_m(t_i)^2\big)^{1/2}}{E(u,Q_i)^{1/2}},
\qquad i=1,2,3,
\]
where $t_i=(i-\frac{1}{2})L$.
We will show that if $\max\{\lambda_1,\lambda_2,\lambda_3\}$ and $\epsilon$ are sufficiently small, 
the energies $E(u,Q_i)$ satisfy the three-circle inequality.

We next consider the case where the quantities $\lambda_i$ 
are bounded from below. 
To this end, we introduce a modified energy $E^*$ as follows:
\[
E^*(u,Q_{a,b})
= \frac{\pi}{2} \int_a^b e^{-2mt} \big(|\hat a_m|^2 + |\hat b_m|^2\big)\, dt.
\]
In the simplest case $m=0$, this corresponds to the 
energy of the average component, namely the $a_0'$ part, while the angular 
energy is controlled by $E(u,Q)-E^*(u,Q)$, i.e.
\[
\int_{Q_{a,b}}|u_\theta|^2\leq (E-E^*)(u,Q_{a,b}).
\]
Consequently, we define quantities 
\[
\nu_i
= \frac{(E-E^*)(u,Q_i)^{1/2}}{e^{-mt_i}\big(\hat a_m(t_i)^2+\hat b_m(t_i)^2\big)^{1/2}},
\qquad i=1,2,3.
\]
When $\nu_i>a$ for some fixed $a>0$, and $\epsilon$ 
is sufficiently small, the difference $(E-E^*)(u,Q_i)$ and $\nu_i$ satisfy the three-circle 
inequality.

We now turn to the analysis on long cylinders. 
Consider $u\colon S^1\times [0,mL]\rightarrow\R^n$ that satisfies \eqref{almost.harmonic.1} and \eqref{almost.harmonic.2}. From an intuitive perspective,  we decompose $S^1\times [0,mL]$ as
\[
S^1\times [0,s_1L]\;\cup\;S^1\times [s_1L,s_2L]\;\cup\; S^1\times [s_2L,s_3L]\;\cup\; S^1\times [s_3L,s_4L]\;\cup\; S^1\times [s_4L,mL].
\]
On the end regions $S^1\times [0,s_1L]$ and $S^1\times [s_4L,mL]$, 
the energy $E$ decays exponentially from the two ends toward the interior. 
On the intermediate regions $S^1\times [s_1L,s_2L]$ and $S^1\times [s_3L,s_4L]$, 
 $E-E^*$ together with  $\nu_i$, 
decays exponentially toward the center. 
In the central region $S^1\times [s_2L,s_3L]$, the energy is essentially 
concentrated in the critical mode, so that $E$ is well approximated by $E^*$, 
while $\nu_i$ is small and $\pi L\lambda_i^2/2$ is close to $1$.

By the above picture, we are led to consider sequences 
of maps $u_k\colon [0,m_k L]\times S^1\to \mathbb{R}^n$, where $m_k\rightarrow+\infty$.
The preceding analysis reveals a 
dichotomy between exponential decay and concentration of the critical 
mode. This dichotomy is quantified in the following main theorem.


\begin{thm}\label{main1}
For any $q\in(0,2)$, we can find $L$ such that the following holds. Let $u_k\colon [0,m_k L]\times S^1\to\mathbb{R}^n$ be a sequence of $(\tau,\epsilon_k)$-almost harmonic functions with 
$\epsilon_k\rightarrow 0$. Define
$$
\Theta_k=E(u_k,([0,L]\cup[(m_k-1)L,m_kL])\times S^1).
$$
Then, after passing to a subsequence, for any $\gamma>0$, either we have
\begin{equation}\label{eqn:trivialenergy}
\lim_{k\rightarrow+\infty}\sum_{i=1}^{m_k}E^\gamma(u_k,Q_i)\leq C\Theta_k^\gamma,
\end{equation}
or we can find integers $0<\mathfrak a_k<\mathfrak b_k<m_k$ with $\mathfrak{b}_k-\mathfrak{a}_k\rightarrow+\infty$ such that 
\begin{equation}\label{eqn:trivialends}
\lim_{k\rightarrow+\infty} (\sum_{i=1}^{\mathfrak{a}_k}+\sum_{i=\mathfrak b_k}^{m_k})E^\gamma(u_k,Q_i)\leq C\Theta_k^\gamma;
\end{equation}
and  
\begin{equation}\label{eq:limit of lambda_i,nu_i}
\lim_{r\to +\infty}\lim_{k\rightarrow+\infty}\max_{\mathfrak a_k+r<i<\mathfrak b_k-r}\left(\left|\frac{\pi L}{2}\lambda_i^2-1\right|+\nu_i+\left|\frac{E^*(u_k,Q_i)}{E(u_k,Q_i)}-1\right|\right)=0.
\end{equation}
Moreover, let \(t_k\in\mathbb Z^+\) be any sequence satisfying 
\[
t_k-\mathfrak a_k\to+\infty
\qquad\text{and}\qquad
\mathfrak b_k-t_k\to+\infty,
\]
and set \(s_k=\left(t_k-\frac12\right)L\). Then, after passing to a further subsequence, there exist
$c_k,A,B\in\mathbb{R}^n$ such that
\[
\begin{aligned}
\frac{e^{-ms_k}\bigl(u_k(s_k+t,\theta)-c_k\bigr)}
{\sqrt{E(u_k,Q_{t_k})}}
&\rightharpoonup
e^{mt}\bigl(A\cos m\theta+B\sin m\theta\bigr)~~~\text{in }W_{\mathrm{loc}}^{2,p}(\mathbb{R}\times S^1),
\quad\text{when }m\neq0,
\end{aligned}
\]
and
\[
\frac{u_k(s_k+t,\theta)-c_k}
{\sqrt{E(u_k,Q_{t_k})}}
\rightharpoonup
At
\quad\text{in }W_{\mathrm{loc}}^{2,p}(\mathbb{R}\times S^1),
\quad\text{when }m=0.
\]
\end{thm}

Theorem \ref{main1} can be used to analyze neck regions arising in blow-up phenomena. 
We first derive the following corollary.
\begin{cor}\label{intr.cor}
Let $u_k$ be as in Theorem \ref{main1} and assume $E(u_k)<\Lambda$. Then
$$
\lim_{T\rightarrow+\infty}\lim_{k\rightarrow+\infty}(E-E^*)(u_k,[T,m_kL-T]\times S^1)=0
$$
\end{cor}
We remark that, to derive similar conclusions for the full energy $E(u)$, 
one needs additional structural properties, such as 
Pohozaev-type identities or conservation laws.

We next consider the case $m=0$ and assume that 
\begin{equation}\label{Ding-Tian reduction}
\sup_i E(u_k,Q_i)\to 0\qquad \text{as}~k\to\infty.   
\end{equation}
We are interested in
the convergence of the images $u_k(S^1\times [0,m_kL])$. 
Such problems arise naturally in the blow-up analysis of neck regions. 

In this case, the images converge, in a suitable sense, 
to a one-dimensional set determined by the averages, i.e., 
the image of $u_k$ is in a small neighborhood 
of 
\[
u_k^*(t)=\frac{1}{2\pi}\int_{S^1}u_k(t,\theta)\,d\theta.
\]
Therefore, it is sufficient to study the behavior of $u_k^*$, or its image $\Gamma_k$. 

We obtain an estimate on its length (see \eqref{eqn:length1} and \eqref{eqn:length2}).
Then we discuss convergence. Both the convergence and the identification 
of the limiting equation require additional assumptions.

First, we assume that for any $\mathfrak a_k\leq i\leq \mathfrak b_k$,
\begin{equation}\label{eqn:dec.Delta} \int_{Q_i}|\Delta u_k|\leq CE(u_k,Q_{i-1}\cup Q_i\cup Q_{i+1}). 
\end{equation}
Under this assumption, the averaged curves $\Gamma_k$ exhibit a compactness property. More precisely, for any choice of points $\mathfrak a_k < c_k < \mathfrak b_k$, the pointed curves $(\Gamma_k, u_k^*(c_k))$ converge, after passing to a subsequence, in the $C^1_{\mathrm{loc}}$ topology.

Furthermore, under an additional structural assumption on $\Delta u_k$, one can identify the equation satisfied by the limiting curve.


\begin{thm}\label{main2}
Let $u_k\colon [0,m_k L]\times S^1\to\mathbb{R}^n$ be a sequence of $(\tau,\epsilon_k)$-almost harmonic functions with 
$\epsilon_k\rightarrow 0$. Assume that \eqref{Ding-Tian reduction} and \eqref{eqn:dec.Delta} hold, and that the lengths of $\Gamma_k$ do not converge to $0$.

Let $\mathfrak a_k < c_k < \mathfrak b_k$, and define
\[
a = \lim_{k\to\infty} \ell(\Gamma_k,[0,c_k]),
\qquad
b = \lim_{k\to\infty} \ell(\Gamma_k,[c_k,m_kL]).
\]
Then, after reparametrization by arc length $s = \int_{c_k}^t |\dot u_k^*|$, the curves $u_k^*(s)$ converge in $C^1_{loc}((-a,b))$ to a limit curve $\gamma$.

Moreover, if $\Delta u_k$ admits a decomposition
\[
\Delta u_k = F_{0,k} + G_k,
\]
where
\[
\int_{Q_i} |G_k|
\le C \epsilon_k E(u_k,Q_{i-1}\cup Q_i\cup Q_{i+1}),
\]
and $F_{0,k}(p,X,Y)=\alpha_k^{ij}(p)X_iY_j$ converges uniformly to $F_0$, then the limit curve $\gamma$ satisfies
\[
-\int_a^b \varphi'(s)\,\frac{d\gamma}{ds}\, ds
=
\int_a^b \varphi(s)
\left(
F_0(\gamma)(\gamma',\gamma')
- \langle F_0(\gamma)(\gamma',\gamma'), \gamma' \rangle \gamma'
\right) ds
\]
for all $\varphi \in \mathcal{D}(a,b)$.
\end{thm}

The preceding results provide a general three-circle framework for neck analysis. In Section \ref{sec:examples}, we then apply this framework to several geometric equations by verifying the almost-harmonic estimates in Definition \ref{defi of almost hm} with respect to the appropriate control energy. The resulting three-circle inequalities provide the key decay estimates on neck regions and, when combined with suitable Pohozaev-type identities or conservation laws, lead to the corresponding energy identities. This approach applies to solutions of Rivi\`ere's equation, conformal immersions with small Willmore energy, harmonic maps on long cylinders without the assumption of weak conformality, \(\alpha\)-harmonic maps, Willmore surfaces.

The three-circle theorem should be viewed more as a method than as a
result for one particular equation, which can be adapted to many different geometric problems.
To illustrate this point, we give two extensions of the main results. To illustrate this point, we consider two extensions of the main results. 

The first is to equations with an additional inhomogeneous perturbation term. A basic example is
$\tau(u)=f$.
Such equations arise naturally in the study of harmonic map type flows, where the time derivative appears as an inhomogeneous term in the elliptic equation at each fixed time. They also appear in the study of Gauss maps of immersed surfaces in Euclidean space, where the equation for the Gauss map contains a term related to the covariant derivative of the mean curvature vector. 
When $u$ is a map from a closed surface into a closed manifold, such a problem has been studied by Qing and Ding \cite{QingTian1997}, Ding and Tian \cite{DingTian1995}. Then Wang, Wei and Zhang \cite{WangWeiZhang2017} proved the energy identity and the no-neck property under a uniform $L^p$ bound on the tension fields for any $p>1$. In their setting, a weak Pohozaev estimate is still available. 
In the flow of Ding, Li and Liu \cite{DingLiLiu2006},  as the flow time tends to infinity, long cylindrical regions appear, and the image of the neck may converge to a geodesic. In their setting, there is no
Pohozaev type inequality.

We will study such a problem in long cylinders without Pohozaev type inequality. We require the perturbation term itself to satisfy a suitable three-circle type decay estimate. We then separate the cylinder into regions where the perturbation is dominant and regions where it is small compared with the energy. We will show that if a nontrivial neck remains, then after a suitable reparametrization its image converges to a geodesic.

The second extension concerns perturbations of the energy itself.
Instead of the Dirichlet energy $E$, we consider a family of energies
of the form
\[
E_k=E+R_k,
\]
where the remainder $R_k$ is small compared with the main energy in
the neck region. Under suitable local estimates, the normalized
contribution of $R_k$ disappears in the contradiction argument for
the three-circle theorem. Therefore, after a small modification of
the definitions of the energy quantities, the main three-circle
estimates and the long-cylinder decomposition remain valid.

A typical example is Lamm's $\epsilon$-harmonic approximation
\[
E_{\epsilon}(u)
=
\int_\Sigma
\left(
|\nabla u|^2+\epsilon|\Delta u|^2
\right)\,dV_g.
\]
After passing to cylindrical coordinates on a neck, the second term
gives a small perturbation of the Dirichlet energy, due to the relation
between the bubbling scale and $\epsilon$. Hence the general results
for the perturbed energies can be applied. We will also show that some
of the conclusions obtained above for $\alpha$-harmonic maps continue
to hold in this setting. In particular, the image of a nontrivial neck,
after a suitable reparametrization, converges to a geodesic. Under
additional symmetry assumptions on the target, the corresponding
conservation laws can further force this limiting geodesic to be
constant.

The remainder of the paper is organized as follows. Section \ref{sec:3-circle Lemma for almost hm} establishes the three-circle inequalities and uses them to prove Theorem \ref{main1} and Corollary \ref{intr.cor}. Section \ref{sec:geodesic analysis} studies the Hausdorff limits of the images on long cylinders and proves Theorem \ref{main2}. Section \ref{sec:examples} presents the geometric applications described above. Finally, Section \ref{sec:3-circle lemma to inho} extends the analysis to perturbed harmonic map equations with \(L^p\) forcing terms, \(p>1\),  and Section~\ref{sec:approx} develops the corresponding theory for approximating energies, with an application to Lamm's biharmonic approximation of harmonic maps.

\section*{Acknowledgements}

The main mathematical part of this work had been completed approximately
four months before the preparation of the present version. AI tools are not used in the development of mathematical
ideas, results, or proofs. They are used only for language editing and
 to check typographical errors. The authors take full responsibility
for the content of the paper.

\section{3-circle Lemma for almost harmonic functions}\label{sec:3-circle Lemma for almost hm}
In this section, we prove the 3-circle lemma. First, we introduce an abstract 
version of this principle, which will be adapted to the quantities considered later.

\begin{defi}\label{def:3circle}
For each $i=1,2,3$, let $\Phi_i\geq 0$ be given. We say that $\Phi_i$ satisfies the 3-circle lemma for $(q,L)$ if 
\begin{equation}\label{eqn:qphi}
\Phi_2\leq e^{-qL} (\Phi_{1}+\Phi_{3}).
\end{equation}
\end{defi}

We will study the 3-circle properties for energies on cylinders for functions that are nearly harmonic. 
Denote $Q_{a,b}=S^1\times [a,b]$. For simplicity, set
$$
Q=S^1\times[0,3L],\quad Q_i=S^1\times[(i-1)L,iL].
$$
Given a map $u\colon Q_{a,b}\rightarrow\R^n$ and $m\in\mathbb{Z}$,  define
$$
E(u,Q_{a,b})=\int_{Q_{a,b}} e^{-2mt}|\nabla u|^2\,dtd\theta.
$$

\subsection{3-circle lemma for harmonic functions}
First, we show that
\begin{lem}
Let \(f\) be harmonic on \(Q\), with expansion 
\begin{align*}
f(t,\theta)
= c+a_0 t + \sum_{k\in\mathbb{Z}\setminus\{0\}} e^{kt}(a_k \cos k\theta
+ b_k\sin k\theta\Big).
\end{align*}
Assume that
\begin{equation*}\label{3-cir.har.ass}
a_m=b_m=0
\end{equation*}
with the convention that $b_0=0$. Then, for any $q\in(0,2]$, there exists $L_0(m,q)>0$ such that if $L>L_0(m,q)$, the sequence $\{E(f,Q_i)\}$ satisfies the 3-circle lemma for $(q,L)$.
\end{lem}

\begin{proof}
A direct computation yields
\begin{eqnarray*}
f_t
&=&a_0 + \sum_{k\neq 0} \Big(
ka_k e^{kt}\cos k\theta
+ kb_k e^{kt}\sin k\theta
\Big)\\
&=&a_0 + \sum_{k=1 }^\infty \Big(
k(a_k e^{kt}-a_{-k} e^{-kt})\cos k\theta
+ k(b_k e^{kt}+b_{-k} e^{-kt})\sin k\theta
\Big),\\
f_\theta
&=& \sum_{k\neq 0} \Big(
- ka_k e^{kt}\sin k\theta
+ kb_k e^{kt}\cos k\theta
\Big)\\
&=& \sum_{k=1}^\infty \Big(
- k(a_k e^{kt}+a_{-k} e^{-kt})\sin k\theta
+ k(b_k e^{kt}-b_{-k} e^{-kt})\cos k\theta
\Big).
\end{eqnarray*}
Then,
we have
\begin{align*}
\int_0^{2\pi} |f_t|^2\,d\theta
&= 2\pi a_0^2
+ \pi \sum_{k=1}^\infty k^2 \Big[
(a_k e^{kt}-a_{-k} e^{-kt})^2
+ (b_k e^{kt}+b_{-k} e^{-kt})^2
\Big],
\end{align*}
\begin{align*}
\int_0^{2\pi} |f_\theta|^2\,d\theta
&= \pi \sum_{k=1}^\infty k^2 \Big[
(a_k e^{kt}+a_{-k} e^{-kt})^2
+ (b_k e^{kt}-b_{-k} e^{-kt})^2
\Big],
\end{align*}
which yields that
\begin{align*}
\int_0^{2\pi} (|f_t|^2 + |f_\theta|^2)\, d\theta
&= 2\pi a_0^2 
+ 2\pi \sum_{k=1}^\infty k^2 \Big(
a_k^2 e^{2kt} + a_{-k}^2 e^{-2kt}
+ b_k^2 e^{2kt} + b_{-k}^2 e^{-2kt}
\Big)\\
&= 2\pi a_0^2
+ 2\pi \sum_{k\neq 0} k^2 \Big
(a_k^2 + b_k^2\Big)e^{2kt}.
\end{align*}
Thus, for $i=1,2,3$, the energy integral can be decomposed as
\begin{equation*}
    \begin{split}
     E(f,Q_i)&=2\pi\int_{(i-1)L}^{iL}\Big( a_0^2e^{-2mt}+\sum_{k\neq 0} k^2(a_k^2+b_k^2)e^{2(k-m)t}\Big)\,dt\\       &:=2\pi\int_{(i-1)L}^{iL}\sum_{k=-\infty}^\infty A_k^2e^{2kt}\,dt.\\
     &:=2\pi\sum_{k=-\infty}^\infty I_i^k.
    \end{split}
\end{equation*}
The assumption $a_m=b_m=0$ implies that $A_0=0$. Consequently, it suffices to analyze the sum by considering two cases.

\textbf{Case 1.} $k<0$.
\begin{equation*}     
I_2^k=\int_{L}^{2L} A_k^2e^{2kt}\,dt=\int_0^LA_k^2e^{2k(\tau+L)}\,d\tau=e^{2kL}I^k_1\leq e^{-2L}I_1^k.
\end{equation*}

\textbf{Case 2.} $k>0$.
\begin{equation*}
I_2^k=\int_{L}^{2L} A_k^2e^{2kt}\,dt=\int_{2L}^{3L}A_k^2e^{2k(\tau-L)}\,d\tau=e^{-2kL}I^k_3\leq e^{-2L}I_3^k.
\end{equation*}
This completes the proof.
\end{proof} 

\subsection{3-circle lemma for almost harmonic functions}

Let $u\in W^{2,p}(Q,\R^n)$.
We write $u$ as
\[
u = a_0(t) + \sum_{k=1}^{\infty} \big( a_k(t)\cos k\theta + b_k(t)\sin k\theta \big).
\]
A direct computation shows that
\begin{align*}
E(u,Q_{a,b})
&= \frac{\pi}{2} \int_a^b \Bigg\{
4e^{-2mt}|a_0'|^2 \\
&\qquad + \sum_{k=1}^\infty e^{-2mt} \Big[ |a_k' - k a_k|^2 + |a_k' + k a_k|^2 \Big] \\
&\qquad + \sum_{k=1}^\infty e^{-2mt} \Big[ |b_k' - k b_k|^2 + |b_k' + k b_k|^2 \Big]
\Bigg\}\, dt.
\end{align*}

Define
\[
\hat b_0 = 0, \quad \hat a_0 = 2a_0', \quad 
\hat a_p = a_{|p|}' + p a_{|p|}, \quad 
\hat b_p = b_{|p|}' + p b_{|p|}, \quad p \in \mathbb{Z}\setminus\{0\}.
\]
Then
\[
E(u,Q_{a,b})
= \frac{\pi}{2} \sum_{p=-\infty}^{\infty} \int_a^b e^{-2mt} \big( |\hat a_p|^2 + |\hat b_p|^2 \big)\, dt.
\]

Moreover,
\[
\Delta u
= a_0''(t)
+ \sum_{k=1}^\infty \big(a_k''(t) - k^2 a_k(t)\big)\cos k\theta
+ \sum_{k=1}^\infty \big(b_k''(t) - k^2 b_k(t)\big)\sin k\theta,
\]
which implies for any $m>0$,
\[
\hat a_m' - m \hat a_m
= a_m'' - m^2 a_m
= \frac{1}{\pi} \int_0^{2\pi} \Delta u \cos m\theta\, d\theta,
\]
and
\[
\hat b_m' - m \hat b_m
= b_m'' - m^2 b_m
= \frac{1}{\pi} \int_0^{2\pi} \Delta u \sin m\theta\, d\theta.
\]
Thus,
\begin{equation*}
\begin{aligned}
(e^{-mt}\hat a_m)' &= \frac{1}{\pi} \int_0^{2\pi} e^{-mt}\Delta u \cos m\theta\, d\theta,\\
(e^{-mt}\hat b_m)' &= \frac{1}{\pi} \int_0^{2\pi} e^{-mt}\Delta u \sin m\theta\, d\theta.
\end{aligned}
\end{equation*}
Consequently,
\begin{equation}\label{diff.Am}
\begin{aligned}
&\Big|e^{-mt}\hat a_m(t) - e^{-ms}\hat a_m(s)\Big|
+ \Big|e^{-mt}\hat b_m(t) - e^{-ms}\hat b_m(s)\Big| \\
&\qquad \le C \int_{Q_{t,s}} e^{-m\tau} |\Delta u|.
\end{aligned}
\end{equation}
The same formulas hold for \(m=0\). For \(m<0\), use \(\cos|m|\theta\) and \(\sin|m|\theta\) in the coefficient identities. The estimate \eqref{diff.Am} is valid for all \(m\in\mathbb Z\), with the interval between \(s\) and \(t\) understood in increasing order.

We define
\[
E^*(u,Q_{a,b})
= \frac{\pi}{2} \int_a^b e^{-2mt} \big(|\hat a_m|^2 + |\hat b_m|^2\big)\, dt,
\]
and set
\[
t_i = \Big(i - \tfrac{1}{2}\Big)L.
\]

\begin{lem}\label{est.ahat}
Let $u$ be a $(\tau,\epsilon)$-almost harmonic function, and suppose $L>1$. Then there exists $C>0$, depending only on $\tau(1/2)$, such that
\[
|e^{-mt_i}\hat a_m(t_i)| + |e^{-mt_i}\hat b_m(t_i)|
\le C \sqrt{E(u,Q_i)}, \quad i=1,2,3,
\]
and
\[
\begin{aligned}
&|e^{-mt_i}\hat a_m(t_i) - e^{-mt_{i+1}}\hat a_m(t_{i+1})| + |e^{-mt_i}\hat b_m(t_i) - e^{-mt_{i+1}}\hat b_m(t_{i+1})|\\
&\quad 
\le C\epsilon \sqrt{E(u,Q_i \cup Q_{i+1})}, \quad i=1,2.
\end{aligned}
\]
\end{lem}

\begin{proof}
The Fourier coefficients satisfy
\begin{equation}\label{trace0}
a_0'(t)
= \frac{1}{2\pi}\int_0^{2\pi} u_t(t,\theta)\, d\theta,
\end{equation}
and for any $j>0$,
\begin{equation}\label{trace1}
a_j'(t)
= \frac{1}{\pi}\int_0^{2\pi} u_t(t,\theta)\cos(k\theta)\, d\theta,
\qquad
a_j(t)
= -\frac{1}{k\pi}\int_0^{2\pi} u_\theta(t,\theta)\sin(k\theta)\, d\theta.
\end{equation}
By the trace embedding inequality and \eqref{almost.harmonic.1}, we have
\begin{align*}
|e^{-mt_i}\hat a_m(t_i)|
&\le C e^{-mt_i}
\left(
\int_{Q_{t_i-\frac{1}{4}L,\; t_i+\frac{1}{4}L}}
|\nabla u|^2
\right)^{1/2} \\
&\le C \sqrt{E(u, Q_{t_i-\frac{1}{4}L,\; t_i+\frac{1}{4}L})}
\le C \sqrt{E(u, Q_i)}.
\end{align*}

Moreover, by \eqref{diff.Am} and \eqref{almost.harmonic.2},
\begin{align*}
\left|e^{-mt_{i+1}}\hat a_m(t_{i+1}) - e^{-mt_i}\hat a_m(t_i)\right|
&\le \frac{e^{-mt_i}}{\pi}
\int_{S^1 \times [t_i,t_{i+1}]} |\Delta u|\, d\theta dt \\
&\le C\epsilon \sqrt{E(u,Q_i\cup Q_{i+1})} .
\end{align*}
The estimates for $\hat b_m$ are analogous.
\end{proof}

We define
$$
\mathcal{A}_i(u)=e^{-mt_i}\sqrt{|\hat a_m|^2(t_i)+|\hat b_m|^2(t_i)},
$$
and
$$
\lambda_i(u)=\frac{\mathcal A_{i}(u)}{\sqrt{E(u,Q_i)}}.
$$


Note that $E(u,Q_i)=0$ implies that $E^*(u,Q_i)=0$ (since it corresponds to the term $p=m$ in the sum of non-negative terms), which in turn yields $\mathcal{A}_i(u)=0$. In this case, the definition of $\lambda_i(u)$ becomes the indeterminate form $\frac{0}{0}$. Thus, we adopt the following convention:
\begin{equation}\label{eq:convention for A_m(u)}
    \lambda_i(u)=0,\qquad \text{when}~\,E(u,Q_i)=0.
\end{equation}

\begin{exa}
When $u$ is harmonic map, we may assume
$$
u=c+c_0t+\sum_{p\in\mathbb{Z}\backslash \{0\}}e^{pt} (c_p \cos p\theta+ c_p' \sin p\theta).
$$
Then for any $k>0$, we have 
$$
a_k(t)=c_ke^{kt}+c_{-k}e^{-kt},
$$
and for any $k\in\mathbb Z$, 
$$
\hat a_k=2kc_ke^{kt}.
$$
We define
$$
\mathcal F(u)=c_me^{mt}\cos m\theta+c_m'e^{mt}\sin m\theta,\qquad m\neq 0,
$$
and
$$\mathcal F(u)=c_0t,\qquad m=0.$$
It is easy to check that
\begin{equation}\label{E.E*.har}
E^*(u,Q_i)=E(\mathcal{F}(u),Q_i),\quad (E-E^*)(u,Q_i)=E(u-\mathcal{F}(u),Q_i).
\end{equation}
\end{exa}

\begin{thm}\label{thm:3 circle thm for u}
Let $q\in(0,2]$, $L>\max\{1,L_0(m,q)\}$, and $\tau$ be any positive function. Then there exist positive constants $\epsilon_0$ and $\beta_0$, depending on $q, L$ and $\tau$, such that for any $\epsilon\in(0,\epsilon_0]$ and any $(\tau,\epsilon)$-almost harmonic function $u$ on $Q$ satisfying $$
\min\{\lambda_1,\lambda_2,\lambda_3\}<\beta_0,
$$
the sequence $\{E(u,Q_i)\}$ satisfies the 3-circle lemma for $(q,L)$.
\end{thm}

\begin{proof}
We argue by contradiction. Suppose the conclusion fails. Then there exist a sequence  $\epsilon_k\to 0$, and a sequence $u_k$ satisfying \eqref{almost.harmonic.1} and \eqref{almost.harmonic.2} and some $i_0\in\{1,2,3\}$ with
$$ \lambda_{i_0}(u_k)\to 0$$
and
$$
0\neq E(u_k,Q_2)>e^{-qL}\left( E(u_k,Q_1) +E(u_k,Q_3)\right).
$$
Set
$$
v_k=\frac{u_k-\bar u_k}{\sqrt{E(u_k,Q_2)}},\s \bar u_k=\vint_{Q_2}u_k.
$$
Then
\begin{equation*}
E(v_k,Q_2)=1,\s 
E(v_k,Q_2)>e^{-qL}\left( E(v_k,Q_1) +E(v_k,Q_3)\right).
\end{equation*}
In particular,
$$
\int_Q|\nabla v_k|^2\,dtd\theta \leq CE(v_k,Q)<C'.
$$
By virtue of \eqref{almost.harmonic.1}, for any fixed $\delta>0$, 
$$
\|\nabla v_k\|_{W^{1,p}(Q_{\delta,3L-\delta})}\leq C(\delta)\|\nabla v_k\|_{L^{2}(Q)}\leq C''.
$$
Moreover, an application of \eqref{almost.harmonic.2} leads to 
$$\|\Delta v_k\|_{L^1(Q_{\delta,3L-\delta})}=\frac{\|\Delta u_k\|_{L^1(Q_{\delta,3L-\delta})}}{\sqrt{E(u_k,Q_2)}}\leq \tau\epsilon_k\|\nabla v_k\|_{L^2(Q)}\to 0\qquad \text{as}~~k\to\infty.$$	
By the Rellich-Kondrachov theorem, up to a subsequence, 
\[
v_k \rightharpoonup v \quad \text{in } W^{2,p}_{\mathrm{loc}}(Q),
\qquad
v_k \to v \quad \text{in } W^{1,2}_{\mathrm{loc}}(Q),
\]
where $v$ is harmonic and satisfies
\[
E(v,Q_2) = 1.
\]
In particular, $v$ is non-constant.
		
Since $\lambda_{i_0}(u_k)\to 0$, we also have $\lambda_{i_0}(v_k)\to 0$. By compact trace embedding theorem and \eqref{trace0}, \eqref{trace1}, 
$$ 
\mathcal A^2_{i_0}(v)=\lim_{k\rightarrow+\infty}\mathcal A^2_{i_0}(v_k)
=\lim_{k\rightarrow+\infty}\lambda^2_{i_0}(v_k)E(v_k,Q_{i_0})= 0.
$$

On the other hand, by Fatou's lemma,
\begin{equation*}
\begin{split}
 e^{-qL}\left( E(v,Q_1) +E(v,Q_3)\right)
 &\leq \liminf_{k\to\infty}e^{-qL}\left( E(v_k,Q_1) +E(v_k,Q_3)\right)\\
 &\leq \lim_{k\to\infty}E(v_k,Q_2)=E(v,Q_2)
\end{split}
\end{equation*}
This yields a direct contradiction to Lemma \ref{3-cir.har.ass} for $v$, thereby completing the proof.
\end{proof}
	
We now turn to the case where $\lambda_i$ remains uniformly bounded below.	
We first show the following:	

\begin{lem}\label{lem:ratio of lambda}
With the same notation as in Theorem \ref{thm:3 circle thm for u}, assume that 
$$
\min\{\lambda_i,\lambda_{i+1}\}\geq \beta_0>0.
$$
Then there exist positive constants $\epsilon_0'$ and $C$ such that if $u$ is a $(\tau,\epsilon)$-almost harmonic function on $Q$
with $\epsilon \le \epsilon_0'$, one has
$$
\frac{1-\frac{C\epsilon_0'}{\beta_0}}{1+\frac{C\epsilon_0'}{\beta_0}}\leq\frac{ \mathcal A_{i+1}(u)}{ \mathcal A_{i}(u)}\leq \frac{1+\frac{C\epsilon_0'}{\beta_0}}{1-\frac{C\epsilon_0'}{\beta_0}}.$$
	\end{lem}

\begin{proof}
By the definition of $\lambda_i$, we have
$$
\sqrt{E(u,Q_i)}=\frac{\mathcal A_{i}(u)}{\lambda_i}.
$$
Hence, by Lemma \ref{est.ahat} and our assumption on $\lambda_i$ and $\lambda_{i+1}$,
\begin{equation}\label{Ai}
    \begin{split}
    &\quad\left|\mathcal A_{i+1}(u)- \mathcal A_{i}(u)\right|\\
&\leq\sqrt{\Big(e^{-mt_{i+1}}\hat a_m(t_{i+1})-e^{-mt_{i}}\hat a_m(t_i)\Big)^2+\left(e^{-mt_{i+1}}\hat b_m(t_{i+1})-e^{-mt_{i}}\hat b_m(t_i)\right)^2}\\
&\leq\left|e^{-mt_{i+1}}\hat a_m(t_{i+1})-e^{-mt_{i}}\hat a_m(t_i)\right|
+ \left|e^{-mt_{i+1}}\hat b_m(t_{i+1})-e^{-mt_{i}}\hat b_m(t_i)\right|\\
&\leq C\epsilon_0'\left(\sqrt{E(u,Q_{i})}+\sqrt{E(u,Q_{i+1})}\right)\\
&\leq C\epsilon_0'\left(\frac{ \mathcal A_{i+1}(u)}{\lambda_{i+1}}+\frac{\mathcal A_{i}(u)}{\lambda_{i}}\right)\\
&\leq C\frac{\epsilon_0'}{\beta_0}\left(\mathcal A_{i}(u)+\mathcal A_{i+1}(u)\right).
    \end{split}
\end{equation}
Since $\lambda_i,\lambda_{i+1}>0$, by our definition of $\lambda_i$ and the convention, $\mathcal A_{i}(u)>0$. The conclusion follows by by dividing both sides by $\mathcal A_{i}(u)$.
	\end{proof}

We now analyze the 3-circle property for the remainder term $E - E^*$.
Define
$$
\nu_i(u)=\frac{\sqrt{(E-E^*)(u,Q_i)}}{\mathcal A_{i}(u)},
$$
where we use the convention that $\nu_i=0$ if $(E-E^*)(u,Q_i)=\mathcal{A}_i(u)=0$, and $\nu_i=\infty$ if $\mathcal{A}_i(u)=0$ but $(E-E^*)(u,Q_i)>0$.

\begin{theorem}\label{thm: 3 circle thm for u-u*}
Let $\beta_0$ be as given in Theorem \ref{thm:3 circle thm for u}. For any $q\in (0,2)$ and $a>0$, suppose that $u$ is a $(\tau,\epsilon)$-almost harmonic function on $Q$ satisfying
$$
\min\{\lambda_1,\lambda_2,\lambda_3\}\geq\beta_0\quad\text{and}\quad\max\{\nu_1,\nu_2,\nu_3\}>a.
$$ 
Then, for any $L>L_1(m,q)=\max\{L_0(m,q),\frac{2\log 2}{2-q}\}$, there exists a positive constant $\epsilon_0''$, depending on $q, L, a$, and $\tau$, such that if $\epsilon \le \epsilon_0''$, 
both $(E-E^*)(u,Q_i)$ and $\nu_i^2$ satisfy the 3-circle lemma for $(q,L)$.
\end{theorem}

\begin{proof}
We first prove the assertion for \(E-E^*\). Suppose the result is false. Then there exist a sequence $\epsilon_k \to 0$, a sequence $u_k$ satisfying \eqref{almost.harmonic.1}-\eqref{almost.harmonic.2}, and some $i' \in \{1,2,3\}$ such that
$$
\min\{\lambda_1(u_k),\lambda_2(u_k),\lambda_3(u_k)\}\geq\beta_0,
$$
\begin{equation*}\label{eq:nu i'}
\nu_{i'}(u_k)>a\quad\text{for some}~i'=1,2,3
\end{equation*}
and
\begin{equation}\label{eq:no 3 circle thm for E}
0\neq (E-E^*)(u_k,Q_2)>e^{-qL}\left((E-E^*)(u_k,Q_1)+(E-E^*)(u_k,Q_3)\right).
\end{equation}
Set
$$
v_k=\frac{u_k-\bar{u}_k}{\sqrt{(E-E^*)(u_k,Q_2)}}.
$$
According to \eqref{eq:no 3 circle thm for E}, we have 
$$
\sum_{i=1}^3(E-E^*)(v_k,Q_i)\leq C\quad\text{and}\quad (E-E^*)(v_k,Q_2)=1.
$$
Since $\nu_{i'}(v_k)=\nu_{i'}(u_k)>a$, we have
$$
\mathcal{A}^2_{i'}(v_k)<\frac{1}{a^2}(E-E^*)(v_k,Q_{i'})\leq \frac{C}{a^2},
$$
and hence $$\mathcal{A}_{i'}(v_k)<C'.$$

Moreover, since $\lambda_i(v_k)=\lambda_i(u_k)\geq\beta_0$, we may assume
$$
\mathcal{A}_{i}(v_k)>0\qquad\text{for}~~i=1,2,3.
$$
By Lemma \ref{lem:ratio of lambda}, for $\epsilon_0''$ sufficiently small,
$$
\frac{1}{2}<\frac{\mathcal{A}_{i+1}(v_k)}{\mathcal{A}_i(v_k)}<2.
$$
It follows that for all $i=1,2,3$, 
$$
\mathcal{A}_i(v_k)<C''.
$$
Using again $\lambda_i(u_k) \ge \beta_0$, we deduce from $\lambda_i(v_k)\geq\beta_0$ that
$$
\sqrt{E(v_k,Q_i)}\leq\frac{1}{\beta_0}\mathcal{A}_i(v_k)<\frac{C''}{\beta_0},
$$
and hence
$$\sum_{i=1}^3\int_{Q_i}|\nabla v_k|^2\,dtd\theta<C.$$
By \eqref{almost.harmonic.2}, there holds
$$
\|\Delta v_k\|_{L^1(Q_{\delta,3L-\delta})}=\frac{\|\Delta u_k\|_{L^1(Q_{\delta,3L-\delta})}}{\sqrt{(E-E^*)(u_k,Q_2)}}\leq \tau\epsilon_k\|\nabla v_k\|_{L^2(Q)}\to 0,\qquad\text{as}~~k\to\infty.$$
Passing to a subsequence, we obtain
\[
v_k \rightharpoonup v \text{ in } W^{2,p}_{\mathrm{loc}}(Q),
\qquad
v_k \to v \text{ in } W^{1,2}_{\mathrm{loc}}(Q),
\]
where $v$ is harmonic and satisfies
$$(E-E^*)(v,Q_2)=1.$$ 

By the weak lower semicontinuity of the energy, 
\begin{equation*}
\begin{split}
&\quad e^{-qL}\Big((E-E^*)(v,Q_1)+(E-E^*)(v,Q_3)\Big)\\
&\leq \liminf_{k\to+\infty}e^{-qL}\Big((E-E^*)(v_k,Q_1)+(E-E^*)(v_k,Q_3)\Big)\\
&\leq\lim_{k\to+\infty}(E-E^*)(v_k,Q_2)\\
&=(E-E^*)(v,Q_2),
\end{split}
\end{equation*}
which contradicts the 3-circle inequality for the harmonic function $v-\mathcal{F}(v)$. This proves the claim for $(E-E^*)(u,Q_i)$.

\medskip
	
Next, we proceed to prove the claim for $\nu_i^2$ by considering two cases.
	
\textbf{Case 1: $(E-E^*)(u,Q_2)=0$.}
By our convention, $\nu_2 = 0$, and there is nothing to prove.
	
\textbf{Case 2: $(E-E^*)(u,Q_2)>0$.}
Observe first that $\min\{\lambda_1,\lambda_2,\lambda_3\}\geq\beta_0 > 0$ implies $\mathcal{A}_i(u)>0$ for $i=1,2,3$.
By Lemma \ref{lem:ratio of lambda}, there exists $\epsilon_0'>0$ such that for any $\epsilon\leq \epsilon_0'$, we have
\[
\frac{1}{\sqrt{2}} \leq \frac{\mathcal A_2(u)}{\mathcal A_1(u)}, \frac{\mathcal A_2(u)}{\mathcal A_3(u)} \leq \sqrt{2}.
\]
Moreover, letting $\bar{q} = \frac{q+2}{2}$, we have just proved the existence of some $\epsilon_0'' > 0$ such that if $\epsilon\leq \epsilon_0''$, then $(E-E^*)(u,Q_i)$ satisfies the 3-circle lemma for $(\bar{q}, L)$, i.e.
\[
(E-E^*)(u,Q_1) +(E-E^*)(u,Q_3)\geq e^{\bar{q}L} (E-E^*)(u,Q_2) > 0.
\]
Defining $\epsilon_0:= \min\{\epsilon_0', \epsilon_0''\}$, it then follows that for all $\epsilon<\epsilon_0$,
\begin{equation*}
    \begin{split}
    \frac{\nu^2_2}{\nu^2_1 + \nu^2_3}
&= \frac{\frac{(E-E^*)(u,Q_2)}{\mathcal A_2^2(u)}}{\frac{(E-E^*)(u,Q_1)}{\mathcal A_1^2(u)} + \frac{(E-E^*)(u,Q_3)}{\mathcal A_3^2(u)}}
\leq \frac{e^{-\bar{q}L}}{\min\left\{\frac{\mathcal A_2^2(u)}{\mathcal A_1^2(u)},\frac{\mathcal A_2^2(u)}{\mathcal A_3^2(u)}\right\}}\\
&\leq 2e^{-\bar{q}L}\leq e^{-qL}.
\end{split}
\end{equation*}
This completes the proof of the claim for $\nu_i^2$.
\end{proof}
	
\subsection{Energy Estimates on Long Cylinders}
We now investigate the behavior of the energy on long cylinders. By iterating the 3-circle inequality obtained above, we derive propagation and decay estimates along the cylinder.
	
\begin{theorem}\label{thm:long.cylinder}
Let $\beta_0$ be as given in Theorem \ref{thm:3 circle thm for u} and $q'\in(0,2)$. Suppose that $u$ is a $(\tau,\epsilon)$--almost harmonic function on $[0, lL] \times S^1$. Then there exist integers $\mathfrak{a}$ and $\mathfrak b$ satisfying 
$$
0 < \mathfrak a <\mathfrak b < l-1
$$ 
such that the following properties hold:
\begin{enumerate}
\item There exists $\epsilon_1>0$ such that if $\epsilon<\epsilon_1$, then
\[
E(u,Q_i) \leq C e^{-iq' L}  E(u,Q_1) \quad \text{for } i \in [1,\mathfrak a],
\]
\[
E(u,Q_i) \leq C e^{-(l-i)q' L}  E(u,Q_l) \quad \text{for } i \in [ \mathfrak b, l],
\]
and
\[
\lambda_i \geq \beta_0 \quad \text{for } i \in [\mathfrak a+1, \mathfrak b-1]
\quad
\text{and}
\quad
\lambda_i < \beta_0 \quad \text{for } i \in [1, \mathfrak a] \cup [\mathfrak b, l].
\]
\item For any integer $2 < r < \dfrac{1}{4}(\mathfrak b-\mathfrak a)$, there exists $\epsilon_2 =\epsilon_2(r)$ such that if $\epsilon<\epsilon_2$, then
\[
\nu_i \leq (1 + \frac{1}{\beta_0^2}) e^{-\frac{q'(r-2)L}{4}}, \quad \forall\, \mathfrak a+r < i < \mathfrak b-r.
\]
\end{enumerate}
\end{theorem}
	
\begin{proof}
(1) Let 
$$
I = \{1 \leq i \leq l \mid \lambda_i \geq \beta_0\}.
$$ 
We first claim that 
\begin{equation}\label{eq:2.17}
\text{either } I = \emptyset\quad\text{or}\quad I = [\xi,\eta]  
\end{equation}
for some integers $1 \leq \xi \leq \eta \leq l$.

Assume there exists $1<i<l$ with $\lambda_i<\beta_0$. By Theorem \ref{thm:3 circle thm for u}, for $\epsilon$ sufficiently small,
\[
\text{either } E(u,Q_i)\le e^{-q'L}E(u,Q_{i+1})
\quad \text{or} \quad
E(u,Q_i)\le e^{-q'L}E(u,Q_{i-1}).
\]

Suppose the first case holds. We claim that $\lambda_{i+1}<\beta_0$. Otherwise $\lambda_{i+1}\ge\beta_0$, and by \eqref{Ai},
\[
|\mathcal A_{i+1}(u)-\mathcal{A}_i(u)|
\le C\epsilon(\sqrt{E(u,Q_i)}+\sqrt{E(u,Q_{i+1})})\leq C\epsilon\sqrt{E(u,Q_{i+1})}. 
\]
Then
\[
\left|1-\frac{\mathcal A_i(u)}{\mathcal A_{i+1}(u)}\right|
\le \frac{C\epsilon}{\lambda_{i+1}}
\le \frac{C\epsilon}{\beta_0}.
\]
Choosing $\epsilon$ small such that
\[
\mathcal A_{i+1}(u)
\le e^{\frac{qL}{2}}\mathcal A_i(u).
\]
Together with $E(u,Q_i)\le e^{-qL}E(u,Q_{i+1})$, we obtain $\lambda_{i+1}<\lambda_i<\beta_0$, which leads to a contradiction.

Hence $\lambda_{i+1}<\beta_0$. Iterating this argument shows $[i,l]\cap I=\emptyset$. A similar argument in the second case yields $[1,i]\cap I=\emptyset$. This proves the claim.

The rest of Part (1) follows by applying Theorem \ref{thm:3 circle thm for u}.

\medskip

(2) Fix $2 < r < \frac{\mathfrak b -\mathfrak a}{4}$ and set
$$
\kappa =  e^{- \frac{(r-2) q' L}{4}} ( 1 + \frac{1}{\beta_0^2} ).
$$

Suppose there exists $\mathfrak a+r<i<\mathfrak b-r$ such that $\nu_i>\kappa$. By the 3-circle inequality for $(E-E^*)(u,Q_i)$, we can choose $\epsilon(r)$, such that
\[
(E-E^*)(u,Q_i)\le e^{-q'L}(E-E^*)(u,Q_{i+1})
\]
or the analogous backward inequality. Assume the former.

Since $\epsilon<\epsilon_2$, and $\min \{ \lambda_{i-1}, \lambda_i, \lambda_{i+1} \} \geq \beta_0$, by the convention \eqref{eq:convention for A_m(u)}, it follows that $\mathcal A_{i+1}(u)>0$. By Lemma \ref{lem:ratio of lambda}, we have
$$
1 - C \frac{\epsilon_2}{\beta_0} \leq \frac{\mathcal{A}_{i+1}(u)}{\mathcal{A}_i(u)} \leq 1 + C \frac{\epsilon_2}{\beta_0}.
$$
Choosing $\epsilon_2$ to be sufficiently small yields
$$
0 < \mathcal{A}_{i+1}(u)\leq e^{\frac{q' L}{4}} \mathcal{A}_i(u).
$$
Combining this with $(E-E^*)(u,Q_i)\leq e^{-q' L} (E-E^*)(u,Q_{i+1})$, we obtain
$$
\nu_{i+1} \geq e^{\frac{q'L}{4}} \nu_i > e^\frac{q'L}{4} \kappa.
$$

Iterating this inequality $r-2$ times gives
\begin{equation}\label{eq:2.20}
	\nu_{i+r-2} > e^{\frac{q' (r-2) L}{4}} \kappa = 1 + \frac{1}{\beta_0^2}. 
\end{equation}
However, since $i'':=i+r-2<\mathfrak b$, we have $\lambda_{i''}\geq \beta_0$, which implies
\begin{equation*}
\begin{split}
\nu_{i''}=\frac{\sqrt{(E-E^*)(u,Q_{i''})}}{\mathcal{A}_{i''}(u)}\leq \frac{\sqrt{E(u,Q_{i''})}}{\mathcal{A}_{i''}(u)}=\frac{1}{\lambda_{i''}}\leq \frac{1}{\beta_0^2}.
	\end{split}
\end{equation*}
This contradicts \eqref{eq:2.20}, which completes the proof of Part (2).
\end{proof}

\subsection{Asymptotic Analysis of Sequences on Long Cylinders}
We now consider a sequence of maps $u_k$ defined on cylinders $S^1 \times [0,m_kL]$ with $m_k \to +\infty$. We study the asymptotic behavior of $\{u_k\}$ using the estimates obtained on finite cylinders.\\

First, we prove Theorem \ref{main1}.

\begin{proof}
Applying Theorem \ref{thm:long.cylinder} to $u_k$, we obtain
integers $\mathfrak a_k<\mathfrak b_k$ such that
\[
\lambda_i(u_k)< \beta_0 \text{ for } 1\leq i\leq\mathfrak a_k, \text{ and }\mathfrak b_k\leq i\leq m_k;
\s
\lambda_i(u_k)\geq \beta_0  \text{ for } \mathfrak a_k+1\leq i\leq\mathfrak b_k-1.
\]

\textbf{Case 1.} Suppose that there exists a constant $\Lambda>0$ such that $\mathfrak b_k-\mathfrak a_k\leq\Lambda$ for all $k$. 

By Theorem \ref{thm:long.cylinder}, for all sufficiently large $k$,
\[
E(u_k,Q_i)\le
\begin{cases}
e^{-qiL}\,\Theta_k, & 1\leq i\leq\mathfrak a_k,\\[\mv]
e^{-(m_k-i)qL}\,\Theta_k, & \mathfrak b_k\leq i\leq m_k.
\end{cases}
\]
Consequently, for every $\gamma>0$,
\[
\left(
\sum_{i=1}^{\mathfrak a_k}
+
\sum_{i=\mathfrak b_k}^{m_k}
\right)
E^\gamma(u_k,Q_i)
\leq C\Theta_k^\gamma.
\]

For $\mathfrak a_k+1\le i\le\mathfrak b_k-1$, Lemma \ref{est.ahat} gives
\[
\begin{aligned}
\mathcal A_{\mathfrak a_k+1}(u_k)
&\leq
\mathcal A_{\mathfrak a_k}(u_k)
+
C\epsilon_k
\left(
\sqrt{E(u_k,Q_{\mathfrak a_k})}
+
\sqrt{E(u_k,Q_{\mathfrak a_k+1})}
\right)\\
&\leq
(\beta_0+C\epsilon_k)
\sqrt{E(u_k,Q_{\mathfrak a_k})}
+
\frac{C\epsilon_k}{\beta_0}
\mathcal A_{\mathfrak a_k+1}(u_k),
\end{aligned}
\]
where we used
\[
\mathcal A_{\mathfrak a_k}(u_k)
=
\lambda_{\mathfrak a_k}(u_k)
\sqrt{E(u_k,Q_{\mathfrak a_k})}
<
\beta_0\sqrt{E(u_k,Q_{\mathfrak a_k})}
\]
and
\[
\sqrt{E(u_k,Q_{\mathfrak a_k+1})}
=
\frac{\mathcal A_{\mathfrak a_k+1}(u_k)}
     {\lambda_{\mathfrak a_k+1}(u_k)}
\leq
\frac{\mathcal A_{\mathfrak a_k+1}(u_k)}
     {\beta_0}.
\]
As \(\epsilon_k\to0\), the preceding inequality yields, for all sufficiently large \(k\),
\[
\mathcal A_{\mathfrak a_k+1}(u_k)
\leq
C\sqrt{E(u_k,Q_{\mathfrak a_k})},
\]
and therefore
\[
E(u_k,Q_{\mathfrak a_k+1})
\leq
\frac{\mathcal A_{\mathfrak a_k+1}(u_k)^2}{\beta_0^2}
\leq
C E(u_k,Q_{\mathfrak a_k})
\leq C\Theta_k.
\]
For
\[
\mathfrak a_k+1\leq i\leq\mathfrak b_k-2,
\]
both $\lambda_i(u_k)$ and $\lambda_{i+1}(u_k)$ are bounded below by
$\beta_0$. Hence Lemma \ref{lem:ratio of lambda} implies, for all
sufficiently large $k$, that
\[
\frac12
\leq
\frac{\mathcal A_{i+1}(u_k)}
     {\mathcal A_i(u_k)}
\leq 2.
\]
Since $\mathfrak b_k-\mathfrak a_k\leq \Lambda$, it follows that
\[
\mathcal A_i(u_k)
\leq
2^\Lambda\mathcal A_{\mathfrak a_k+1}(u_k)
\]
for every
$\mathfrak a_k+1\leq i\leq\mathfrak b_k-1$.
Using $\lambda_i(u_k)\geq\beta_0$, we obtain
\[
\begin{aligned}
E^\gamma(u_k,Q_i)=\frac{\mathcal A_i^{2\gamma}(u_k)}
     {\lambda_i^{2\gamma}(u_k)}
     \leq\frac{2^{2\gamma \Lambda}}{\beta_0^{2\gamma}}
\mathcal A_{\mathfrak a_k+1}^{2\gamma}(u_k)
\leq C\Theta_k^\gamma.
\end{aligned}
\]
Combining the preceding estimates yields
\[
\sum_{i=1}^{m_k}E(u_k,Q_i)^\gamma
\leq C\Theta_k^\gamma.
\]
This proves \eqref{eqn:trivialenergy}.\\

\textbf{Case 2.} We assume $\mathfrak b_k-\mathfrak a_k\rightarrow+\infty$.

The first part of Theorem \ref{thm:long.cylinder} gives
\[
\begin{aligned}
\left(
\sum_{i=1}^{\mathfrak a_k}
+
\sum_{i=\mathfrak b_k}^{m_k}
\right)
E(u_k,Q_i)^\gamma
&\leq
C\Theta_k^\gamma
\left(
\sum_{i=1}^{\mathfrak a_k}e^{-i\gamma qL}
+
\sum_{i=\mathfrak b_k}^{m_k}
e^{-(m_k-i)\gamma qL}
\right)\\
&\leq C\Theta_k^\gamma.
\end{aligned}
\]
This proves \eqref{eqn:trivialends}.

We next study the interior region. For any fixed integer $r>2$, the assumptions \(\mathfrak b_k-\mathfrak a_k\to+\infty\) and \(\epsilon_k\to0\) ensure that, for all sufficiently large $k$, 
\[
r<\frac14(\mathfrak b_k-\mathfrak a_k)
\quad\text{and}\quad
\epsilon_k<\epsilon_2(r).
\]
Therefore, the second part of Theorem
\ref{thm:long.cylinder} gives
\[
\nu_i(u_k)
\leq
\left(1+\frac{1}{\beta_0^2}\right)
e^{-\frac{q'(r-2)L}{2}}\qquad \mathfrak a_k+r<i<\mathfrak b_k-r.
\]
It follows that
\[
\lim_{r\to\infty}\lim_{k\to\infty}
\max_{\mathfrak a_k+r<i<\mathfrak b_k-r}
\nu_i(u_k)=0.
\]

Together with the definition of the central region, Lemma \(\ref{est.ahat}\) gives
$$
\beta_0
\leq
\lambda_i(u_k)
\leq C,
\qquad
\mathfrak a_k<i<\mathfrak b_k.
$$
Hence,
$$
0
\leq
1-\frac{E^*(u_k,Q_i)}{E(u_k,Q_i)}
=
\frac{(E-E^*)(u_k,Q_i)}{E(u_k,Q_i)}
=
\nu_i(u_k)^2\lambda_i(u_k)^2
\leq C\nu_i(u_k)^2.
$$
It follows that
$$
\lim_{r\to\infty}\lim_{k\to\infty}
\max_{\mathfrak a_k+r<i<\mathfrak b_k-r}
\left|
\frac{E^*(u_k,Q_i)}{E(u_k,Q_i)}-1
\right|
=0.
$$

It remains to estimate \(\lambda_i\). For adjacent indices in the central region, Lemma \(\ref{lem:ratio of lambda}\) gives
$$
\frac{1}{C}
\leq
\frac{\mathcal A_i(u_k)}
{\mathcal A_{i-1}(u_k)}
\leq C
$$
for all sufficiently large \(k\). Since
$$
\frac{\sqrt{E(u_k,Q_i)}}
{\sqrt{E(u_k,Q_{i-1})}}
=
\frac{\lambda_{i-1}(u_k)}
{\lambda_i(u_k)}
\frac{\mathcal A_i(u_k)}
{\mathcal A_{i-1}(u_k)},
$$
the preceding bounds imply
$$
\frac{1}{C}
\leq
\frac{E(u_k,Q_i)}
{E(u_k,Q_{i-1})}
\leq C.
$$
Therefore, whenever \(i-1,i,i+1\) lie in the central region, \(\eqref{diff.Am}\), \(E^*\leq E\), and the adjacent energy comparison give
$$
\begin{aligned}
\left|
\frac{\pi L}{2}\lambda_i^2(u_k)
-
\frac{E^*(u_k,Q_i)}{E(u_k,Q_i)}
\right|
&=
\frac{1}{E(u_k,Q_i)}
\left|
\frac{\pi L}{2}\mathcal A_i^2(u_k)
-
E^*(u_k,Q_i)
\right|\\
&\leq
C\epsilon_k
\frac{
E^*(u_k,Q_{i-1})
+
E^*(u_k,Q_i)
+
E^*(u_k,Q_{i+1})
}{
E(u_k,Q_i)
}\\
&\leq C\epsilon_k.
\end{aligned}
$$
Consequently,
$$
\lim_{r\to\infty}\lim_{k\to\infty}
\max_{\mathfrak a_k+r<i<\mathfrak b_k-r}\left|
\frac{\pi L}{2}\lambda_i(u_k)^2-1
\right|=0.
$$

Finally, let $t_k$ be any sequence of integers satisfying \(t_k-\mathfrak a_k\to+\infty\), and \(\mathfrak b_k-t_k\to+\infty\). Set
\(s_k=\left(t_k-\frac12\right)L\) and \(c_k=\vint_{Q_{t_k}}u_k\), and define
\[
v_k(t,\theta)
=
\frac{
e^{-ms_k}\bigl(u_k(s_k+t,\theta)-c_k\bigr)
}{
\sqrt{E(u_k,Q_{t_k})}
}.
\]
For \(\widetilde Q_j=\left[
\left(j-\frac12\right)L,
\left(j+\frac12\right)L
\right]\times S^1\), a direct change of variables gives
\[
E(v_k,\widetilde Q_0)=1.
\]

For every fixed integer $J>0$, the indices $t_k+j$, $|j|\leq J$,
belong to the central region for all sufficiently large $k$.
The bounds
\[
\beta_0\leq\lambda_{t_k+j}(u_k)\leq C
\]
and Lemma \ref{lem:ratio of lambda} implies
\begin{equation}\label{eq:local energy comparability}
    \frac{1}{C_J}\leq
\frac{E(u_k,Q_{t_k+j})}
     {E(u_k,Q_{t_k})}
\leq C_J,
\qquad |j|\leq J.
\end{equation}
The defining estimates for $(\tau,\epsilon_k)$--almost harmonic
functions, together with the Poincar\'e inequality, therefore yield
\[
\|v_k\|_{W^{2,p}(K)}\leq C_K
\]
for every compact set
$K\subset\subset\mathbb R\times S^1$. Moreover,
\[
\|\Delta v_k\|_{L^1(K)}
\leq C_K\epsilon_k\to0.
\]
After passing to a further subsequence and using a diagonal argument,
there exists a harmonic function
$v\colon\mathbb R\times S^1\to\mathbb R^n$ such that
\[
v_k\rightharpoonup v
\quad\text{in }
W_{\mathrm{loc}}^{2,p}(\mathbb R\times S^1).
\]
In particular,
\[
E(v,\widetilde Q_0)=1,
\]
so $v$ is non-constant.

For every fixed $j\in\mathbb Z$, the assumptions on $t_k$ and the
interior estimates above give
\[
\frac{(E-E^*)(u_k,Q_{t_k+j})}
     {E(u_k,Q_{t_k+j})}
\to0.
\]
Combining this with \(\eqref{eq:local energy comparability}\) and the identity
$$ \begin{aligned} (E-E^*)(v_k,\widetilde Q_j) &= \frac{(E-E^*)(u_k,Q_{t_k+j})} {E(u_k,Q_{t_k})}\\ &= \frac{(E-E^*)(u_k,Q_{t_k+j})} {E(u_k,Q_{t_k+j})} \frac{E(u_k,Q_{t_k+j})} {E(u_k,Q_{t_k})}, \end{aligned} $$
we obtain
$$ (E-E^*)(v_k,\widetilde Q_j)\to 0. $$
Since \(v_k\to v\) strongly in \(W_{\mathrm{loc}}^{1,2}\), it follows that
\[
(E-E^*)(v,\widetilde Q_j)=0
\qquad\text{for every }j\in\mathbb Z.
\]
Thus $v$ contains only the critical Fourier mode.

Since $c_k$ is the average of $u_k$ over $Q_{t_k}$, the average of
$v$ over $\widetilde Q_0$ is zero. Therefore, when $m\neq0$, there
exist $A,B\in\mathbb R^n$, not both zero, such that
\[
v(t,\theta)
=
e^{mt}
\left(A\cos m\theta+B\sin m\theta\right).
\]
When $m=0$, there exists a nonzero vector $A\in\mathbb R^n$ such that
\[
v(t,\theta)=At.
\]
This completes the proof.
\end{proof}

Now, we prove Corollary \ref{intr.cor}.

\begin{proof}
Let \(\epsilon>0\). Since \(E(u_k)\leq\Lambda\), we can choose integers
\[
1\leq c_k\leq\left\lfloor\frac{\Lambda}{\epsilon}\right\rfloor+1,
\qquad
\Theta'_k:=E(u_k,Q_{c_k}\cup Q_{m_k-c_k+1})<\epsilon.
\]
Apply the estimates in the proof of Theorem \ref{main1} to the restricted domain
\[
[(c_k-1)L,(m_k-c_k+1)L]\times S^1,
\]
and let \(\mathfrak a_k,\mathfrak b_k\) be the corresponding indices. The exponential end estimates give
\[
\left(\sum_{i=c_k}^{\mathfrak a_k}+\sum_{i=\mathfrak b_k}^{m_k-c_k+1}\right)
E(u_k,Q_i)\leq C\epsilon.
\]
For each fixed integer \(r>2\), Lemma \ref{est.ahat} and \ref{lem:ratio of lambda} give, for all sufficiently large \(k\),
\[
E(u_k,Q_{\mathfrak a_k+j})+E(u_k,Q_{\mathfrak b_k-j})
\leq C_r\epsilon,
\qquad
1\leq j\leq\min\{4r,\mathfrak b_k-\mathfrak a_k-1\}.
\]
Here \(C_r\) is independent of \(k\) and \(\epsilon\).

If \(\mathfrak b_k-\mathfrak a_k\leq4r\), these estimates control the entire central energy by \(C_r\epsilon\). If \(\mathfrak b_k-\mathfrak a_k>4r\), Theorem \ref{thm:long.cylinder} and Lemma \ref{est.ahat} give
\[
\begin{aligned}
(E-E^*)(u_k,Q_i)
&=\nu_i(u_k)^2\lambda_i(u_k)^2E(u_k,Q_i)\\
&\leq\delta_rE(u_k,Q_i),
\qquad \mathfrak a_k+r<i<\mathfrak b_k-r,
\end{aligned}
\]
where \(\delta_r\to0\) as \(r\to\infty\). The first and last \(r\) central blocks have total energy at most \(C_r\epsilon\), while the remaining contribution is at most \(\delta_r\Lambda\). Consequently, in both cases,
\[
(E-E^*)\bigl(u_k,[(c_k-1)L,(m_k-c_k+1)L]\times S^1\bigr)
\leq C_r\epsilon+\delta_r\Lambda
\]
for all sufficiently large \(k\).
Since \(c_k\leq\lfloor\Lambda/\epsilon\rfloor+1\), taking \(T\) large enough gives
\[
[T,m_kL-T]\times S^1
\subset
[(c_k-1)L,(m_k-c_k+1)L]\times S^1.
\]
It follows that
\[
0\leq
\lim_{T\to\infty}\lim_{k\to\infty}
(E-E^*)(u_k,[T,m_kL-T]\times S^1)
\leq C_r\epsilon+\delta_r\Lambda.
\]
Letting \(\epsilon\to0\) with \(r\) fixed, and then \(r\to\infty\), gives zero. 
\end{proof}

\section{Hausdorff limits and the proof of Theorem \ref{main2}}\label{sec:geodesic analysis}
In this section, we set $m=0$ and let $u_k\colon S^1\times[0,m_kL]\rightarrow\R^n$ be a sequence of $(\tau, \epsilon_k)$-almost harmonic functions with $\epsilon_k\rightarrow 0$. Furthermore, we assume that
$$
\sup_i E(u_k,Q_i)\rightarrow 0\qquad \text{as}~k\to\infty.
$$
Specifically, we write $u_k$ as
$$
u_k=a_{k,0}(t)+\sum_{j=1}^\infty \left(a_{k,j}(t)\cos j\theta+b_{k,j}(t)\sin j\theta\right).
$$
Then
$$
E^*(u_k,Q_{a,b})=2\pi\int_a^b |(a_{k,0})'|^2\,dt.
$$

Define a family of curves $\Gamma_k$ as follows
\[
\gamma_k: \,u_k^*(t) = \vint_{S^1} u_k(t, \theta) \, d\theta=a_{k,0}(t),\quad t\in[0, m_kL].
\]
By trace embedding, $u_k(t,\theta)\in W^{1,2}(S^1)$ for each fixed $t$, and
it is easy to check that 
$$
d_H(u_k(S^1\times[L,(m_k-1)L]),u_k^*([L,(m_k-1)L]))\rightarrow 0.
$$

In this section, we study the behavior of the averaged curves $\Gamma_k$. The analysis has three main objectives.
First, we establish an estimate for the length of the averaged curves.
Second, under suitable additional assumptions, we prove the convergence of these curves. We emphasize that this is not the convergence of the maps $u_k$ themselves, but rather the convergence of the associated curves parametrized by arc length. This can be naturally interpreted as a convergence in the varifold sense.
Third, again under additional assumptions, we identify the limiting equation satisfied by the resulting curve.

Let $\mathfrak a_k$, $\mathfrak b_k$ be as in Theorem \ref{main1}. 
First, we prove the following estimates:
\begin{lem}\label{lem:est.a0}
There exist \(\delta_r\to0\) as \(r\to\infty\) such that, for each fixed \(r>2\) and all sufficiently large \(k\), the following estimates hold for every \(\mathfrak a_k+r<i<\mathfrak b_k-r\):
\begin{equation}\label{eqn:a0E1}
\left|\frac{E(u_k,Q_i)}{2\pi L|a_{k,0}'(t_i)|^2}-1\right|+
\frac{\int_{Q_i}|\nabla_\theta u_k|^2}{|a_{k,0}'(t_i)|^2}
\leq \delta_r,
\end{equation}
and
\begin{equation}\label{eqn:a0E2}
\left|\frac{\int_{Q_i}|\nabla u_k|}{2\pi L|a_{k,0}'(t_i)|}-1\right|+
\frac{\int_{Q_i}|\nabla_\theta u_k|}{|a_{k,0}'(t_i)|}+\left\|\frac{\dot u_k^*(t)}{a_{k,0}'(t_i)}-1\right\|_{C^0([(i-1)L,iL])}
\leq \delta_r.
\end{equation}
\end{lem}

\begin{proof}
The first term in \eqref{eqn:a0E1} follows directly from 
$$
\lim_{r\to+\infty}\lim_{k\to+\infty}\max_{\mathfrak a_k+r\leq i\leq \mathfrak b_k-r}\left|\frac{\pi L}{2}\lambda_i^2-1\right|= 0.
$$
Moreover, since 
$$
\int_{Q_i}|\nabla_\theta u_k|^2\leq E(u_k,Q_i)-E^*(u_k,Q_i),
$$
the second term in \eqref{eqn:a0E1} follows from 
$$
\lim_{r\to+\infty}\lim_{k\to+\infty}\max_{\mathfrak a_k+r\leq i\leq \mathfrak b_k-r}\nu_i=0.
$$

For  $t\in[(i-1)L,iL]$, by Theorem \ref{main1}, \eqref{diff.Am} and Lemma \ref{lem:ratio of lambda},
\begin{eqnarray*}
\left|\frac{1}{2\pi}\int_{0}^{2\pi}\frac{\partial u_k}{\partial t}-a_{k,0}'(t_i)\right|&\leq& C\epsilon_k
\sqrt{E(u_k,Q_{i-1}\cup Q_i\cup Q_{i+1})}\\
&\leq& C\epsilon_k(|a_{k,0}'(t_{i-1})|
+|a_{k,0}'(t_{i})|+|a_{k,0}'(t_{i+1})|)\\
&\leq&
C\epsilon_k|a_{k,0}'(t_i)|.
\end{eqnarray*}
Hence
$$
\left|\frac{\dot u_k^*(t)}{a_{k,0}'(t_i)}-1\right|\leq C\epsilon_k.
$$

Next, by Cauchy-Schwarz,
$$
\frac{\int_{Q_i}|\nabla_\theta u_k|}{|a_{k,0}'(t_i)|}\leq
\frac{\sqrt{2\pi L\int_{Q_i}|\nabla_\theta u_k|^2}}{|a_{k,0}'(t_i)|},
$$
which yields the second term in \eqref{eqn:a0E2}.

Finally,
\[
\int_{Q_i} \left| \frac{\partial u_k}{\partial t} - a_{k,0}'(t) \right|^2
= \int_{Q_i}\left( \left| \frac{\partial u_k}{\partial t} \right|^2 - |a_{k,0}'(t)|^2\right)
\le E(u_k,Q_i) - E^*(u_k,Q_i),
\]
and therefore
\[
\int_{Q_i} \left| \frac{\partial u_k}{\partial t} - a_{k,0}'(t) \right|^2
\le4 \nu_i^2 |a_{k,0}'(t_i)|^2.
\]
It follows that
\[
\left|
\int_{Q_i} \left| \frac{\partial u_k}{\partial t} \right|
- 2\pi L |a_{k,0}'(t_i)|
\right|
\le
\int_{Q_i} \left| \frac{\partial u_k}{\partial t} - a_{k,0}'(t_i) \right|
\le
C \nu_i |a_{k,0}'(t_i)|.
\]

Combining this with
\[
\int_{Q_i} \left| \frac{\partial u_k}{\partial t} \right|
\le \int_{Q_i} |\nabla u_k|
\le
\int_{Q_i} \left| \frac{\partial u_k}{\partial t} \right|
+
\int_{Q_i} \left| \frac{\partial u_k}{\partial \theta} \right|,
\]
we obtain the first term in \eqref{eqn:a0E2}. This completes the proof.
\end{proof}

Recall that
\[
\ell(\Gamma_k,[(i-1)L,iL])
= \int_{(i-1)L}^{iL} |\dot u_k^*(t)|\,dt.
\]
By the Cauchy–Schwarz inequality, we obtain
\[
\ell(\Gamma_k,[(i-1)L,iL])
\le C L \sqrt{E(u_k,Q_i)}.
\]
Hence, by \eqref{eqn:trivialends},
\[
\ell(\Gamma_k,[0,\mathfrak a_kL]\cup [\mathfrak b_kL,m_kL])
\le C \left( \sum_{i=1}^{\mathfrak a_k} + \sum_{i=\mathfrak b_k}^{m_k} \right) \sqrt{E(u_k,Q_i)}
\longrightarrow 0.
\]
For each fixed integer \(r>2\), the sum above still tends to zero with \(\mathfrak a_k,\mathfrak b_k\) replaced by \(\mathfrak a_k+r,\mathfrak b_k-r\), since only \(2r\) blocks are added and \(\sup_iE(u_k,Q_i)\to0\).

On the other hand, by the third term in \eqref{eqn:a0E2},
\[
\left|
\frac{\int_{(i-1)L}^{iL} |\dot u_k^*(t)|\,dt}{|a_{k,0}'(t_i)| L}
- 1
\right|
< \delta_r,\quad\mathfrak a_k+r<i<\mathfrak b_k-r.
\]
Summing over $i$ yields
\begin{equation}\label{eqn:length1}
\lim_{k\to+\infty} \ell(\Gamma_k)
=
\lim_{k\to+\infty} \sum_{i=1}^{m_k} |a_{k,0}'(t_i)|\,L.
\end{equation}
Moreover, by Lemma \ref{lem:est.a0}, we also have
\begin{equation}\label{eqn:length2}
\lim_{k\to+\infty} \ell(\Gamma_k)
=\sqrt{\frac{L}{2\pi}}
\lim_{k\to+\infty} \sum_{i=1}^{m_k} \sqrt{E(u_k,Q_i)}.
\end{equation}

\begin{pro}\label{prop 3.2}
Let $u_k$ be as in Theorem \ref{main1}, and assume that
\[
\sup_i E(u_k,Q_i) \to 0.
\]
If the lengths $\ell(\Gamma_k)$ are uniformly bounded, then
\[
\lim_{k\to+\infty} E(u_k, S^1 \times [0,m_kL]) = 0.
\]
\end{pro}

\begin{proof}
By \eqref{eqn:length2}, bounded length implies that \(\sum_i\sqrt{E(u_k,Q_i)}\) is bounded. Hence
\[
\sum_i E(u_k,Q_i)\leq
\left(\sup_i E(u_k,Q_i)\right)^{1/2}\sum_i\sqrt{E(u_k,Q_i)}\to0.
\]
\end{proof}

Next, we study the convergence of the curves $\Gamma_k$ under the following assumption: 
\begin{equation}\label{eqn:Fk}
\int_{Q_i}|F_k(t,\theta,u_k,\nabla u_k)|\leq CE(u_k,Q_{i-1}\cup Q_i\cup Q_{i+1}).
\end{equation}
We will show that, after passing to a subsequence, $\gamma_k$ converges to a limiting curve $\gamma$ in $C^1_{\mathrm{loc}}$. 

For convenience, denote
$$
\Delta u_k=F_k(x,u_k,\nabla u_k),\s \dot{u}_k^* = \frac{\partial u_k^*}{\partial t}\s \text{and}\s \ddot{u}_k^* = \frac{\partial^2 u_k^*}{\partial t^2}.
$$
First, observe that
\begin{equation*}
\begin{split}
\ddot{u}_k^*(t) &= \vint_{S^1} \frac{\partial^2 u_k}{\partial t^2}(t, \theta) \, d\theta
=\vint_{S^1} \left(\frac{\partial^2 u_k}{\partial t^2}(t, \theta)+ \frac{\partial^2 u_k}{\partial \theta^2}(t, \theta) \right)\,d\theta\\
&=\vint_{S^1} \Delta u_k(t, \theta) \, d\theta
=\vint_{S^1} F_k(x,u_k,\nabla u_k)\, d\theta.
\end{split}
\end{equation*}

We reparameterize the curve $\gamma_k$ by its arc length
\[
s(t) = \int_{c_k}^{t} \left| \dot{u}_k^*(\tau) \right| d\tau,
\]
where $\mathfrak a_k<c_k<\mathfrak b_k$. Note that $\frac{dt}{ds}=\frac{1}{|\dot{u}_k^*|}$, and thus the chain rule gives
\begin{equation}\label{eq:second order for uk*}
\begin{split}
\frac{d^2 u_k^*}{ds^2}
&= \frac{1}{|\dot{u}_k^*|}\frac{d}{dt}\left(\frac{1}{|\dot{u}_k^*|}\frac{du_k^*}{dt}\right)\\
&= \frac{\ddot{u}_k^*}{|\dot{u}_k^*|^2}-\frac{ \langle \ddot{u}_k^*,\ \dot{u}_k^* \rangle }{ |\dot u_k^*|^4 } \dot{u}_k^* \\
&=\frac{1}{ |\dot{u}_k^*|^2} \vint_{S^1} F_k\, d\theta- \frac{\dot{u}_k^*}{|\dot u_k^*|^4} \left\langle \vint_{S^1} F_k\,d\theta,\ \dot{u}_k^* \right\rangle.
				\end{split}
\end{equation}

Choose $c_k'$ such that 
$$
\int_{c_k}^{c_k'}|\dot u_k^*|\in (s_0,2s_0).
$$
Then
\begin{eqnarray*}
\int_{0}^{s_0}\left|\frac{d^2u_k^*}{ds^2}\right|ds
&\leq&\int_{c_k}^{c_k'}\int_0^{2\pi}\left|\frac{1}{ |\dot{u}_k^*|^2}  F_k\, - \frac{\dot{u}_k^*}{|\dot u_k^*|^4} \langle  F_k,\ \dot{u}_k^*\rangle\right||\dot u_k^*|dtd\theta\\
&=&\int_{Q_{c_k,c_k'}}O\left(\frac{|F_k|}{|\dot u_k^*|}\right)dtd\theta.
\end{eqnarray*}
Applying Lemma \ref{lem:est.a0}, \eqref{eqn:Fk} and Lemma \ref{lem:ratio of lambda}, we obtain
\begin{eqnarray*}
\int_{Q_{c_k,c_k'}}\frac{|F_k|}{|\dot u_k^*|}dtd\theta
&\leq &C\sum_{i=c_k}^{c_k'}\int_{Q_i}\frac{|F_k|}{|\dot u_k^*|}dtd\theta\\
&\leq& C
\sum_{i=c_k}^{c_k'}\int_{Q_i}\frac{|F_k|}{|a_{k,0}'(t_i)|}dtd\theta\\
&\leq&
C\sum_{i=c_k}^{c_k'}\frac{1}{|a_{k,0}'(t_i)|}\int_{Q_{i-1}\cup Q_i\cup Q_{i+1}}|\nabla u_k|^2dtd\theta\\
&\leq& C\sum_{i=c_k}^{c_k'}\frac{1}{|a_{k,0}'(t_i)|}(|a_{k,0}'(t_{i-1})|^2+|a_{k,0}'(t_{i})|^2+|a_{k,0}'(t_{i+1})|^2)\\
&\leq&C\sum_{i=c_k}^{c_k'}(|a_{k,0}'(t_{i-1})|+|a_{k,0}'(t_{i})|+|a_{k,0}'(t_{i+1})|)\\
&\leq &Cs_0.
\end{eqnarray*}
Therefore, $\|u_k^*\|_{W^{2,1}(0,s_0)}$ is uniformly bounded. In particular, the second variation of $\gamma_k$ (viewed as a varifold) is locally bounded. It follows that, after passing to a subsequence,
\[
(\gamma_k,\gamma_k(c_k)) \to (\gamma,P) \quad \text{in } C^1_{\mathrm{loc}},
\]
where $P$ is the limit of $u_k^*(c_k)$.\\

Next, we assume that
$$
F_k=F_{0,k}+G_k,
$$
where 
$$
\int_{Q_i}|G_k|\leq C\epsilon_kE(u_k,Q_{i-1}\cup Q_i\cup Q_{i+1}),
$$		
and
$$F_{0,k}(p,X,Y)=\alpha_k^{ij}(p)X_iY_j.$$
We assume that $F_{0,k}$ is continuous and converges
uniformly to $F_0$.

For any $\varphi\in\mathcal{D}(0,s_0)$, we have
\begin{eqnarray*}
\int_0^{s_0}\varphi(s)\frac{d^2u_k^*}{ds^2}&=&
\int_{c_k}^{c_k'}\varphi(s(t))\left(\frac{1}{|\dot u_k^*|}\vint_{S^1}  F_k\, d\theta- \frac{\dot{u}_k^*}{|\dot u_k^*|^3} \left\langle \vint_{S^1} F_k\,d\theta,\ \dot{u}_k^* \right\rangle\right)dt\\
&=&\int_0^{s_0}\varphi \left(F_{0,k}(u_k^*)(\frac{du_k^*}{ds})(\frac{du_k^*}{ds})-\left\langle F_{0,k}(u_k^*)(\frac{du_k^*}{ds})(\frac{du_k^*}{ds}),\frac{du_k^*}{ds}\right\rangle\frac{du_k^*}{ds}\right)\\
&&+\int_{Q_{c_k,c_k'}}O(\|F_{0,k}(u_k(t,\cdot))-F_{0,k}(u_k^*)\|_{L^\infty(S^1)})\frac{|\nabla u_k|^2}{|\dot u_k^*|}\\
&&+\int_{Q_{c_k,c_k'}}\frac{1}{|\dot u_k^*|}\left(O(|\nabla u_k||\nabla_\theta u_k|)+O(|\nabla u_k||\frac{\partial u_k}{\partial t}-\dot u_k^*|)\right)\\
&&+\int_{Q_{c_k,c_k'}}\frac{|G_k|}{|\dot u_k^*|}.
\end{eqnarray*}

Using the same estimates as before, we obtain
$$
\int_{Q_{c_k,c_k'}}\frac{|\nabla u_k|^2}{|\dot u_k^*|}=O(1)s_0,\s \int_{Q_{c_k,c_k'}}\frac{|G_k|}{|\dot u_k^*|}=o(1)s_0,
$$
and
$$
\int_{Q_{c_k,c_k'}}\frac{|\nabla u_k||\nabla_\theta u_k|}{|\dot u_k^*|}=o(1)s_0,\s
\int_{Q_{c_k,c_k'}}\frac{|\nabla u_k||\frac{\partial u_k}{\partial t}-\dot u_k^*|}{|\dot u_k^*|}=o(1)s_0.
$$
Passing to the limit and using the uniform convergence of $F_{0,k}$, we obtain
$$
-\int_0^{s_0}\varphi'\frac{d\gamma}{ds}=\int_0^{s_0}\varphi \left(F_0(\gamma)(\frac{d\gamma}{ds})(\frac{d\gamma}{ds})-\left\langle F_0(\gamma)(\frac{d\gamma}{ds})(\frac{d\gamma}{ds}),\frac{d\gamma}{ds}\right\rangle\frac{d\gamma}{ds}\right).
$$
This completes the proof of Theorem \ref{main2}.

\section{Applications: energy identities for geometric PDEs}\label{sec:examples}

\subsection{Rivi\`ere's equation}
Let $u_k \in W^{1,2}(D_1, \mathbb{R}^n)$ be weak solutions of 
 $$\Delta u_k=\Omega_k \cdot\nabla u_k,\qquad\Omega_k\in L^2(D_1,so_n\otimes \mathbb{R}^2)$$
 satisfying
 $$\int_{D_1}(|\nabla u_k|^2+|\Omega_k|^2)\,dx\leq \Lambda.$$
The angular energy identity for sequences of such solutions was established by Laurain and Rivi\`ere \cite{LaurainRiviere2014} utilizing a Lorentz duality argument. We prove angular energy vanishing on the necks using Corollary \(\ref{intr.cor}\), and hence the 3-circle lemma. Throughout this subsection, we take $m=0$ and fix $p\in(1,2)$.

By the bubble extraction and neck decomposition described in Laurain–Rivi\`ere \cite[Section 3.1]{LaurainRiviere2014}, it suffices to prove angular energy vanishing on each of the finitely many neck annuli. After translation and rescaling, we write the neck annulus as 
\[
D_1\setminus D_{r_k},
\qquad r_k\to 0.
\]

Set $T_k=\log(1/r_k)$ and define
\[
v_k(t,\theta)
=
u_k\!\left(e^{-t},\theta\right).
\]
Then \(v_k\) satisfies
$\Delta v_k=\widetilde\Omega_k\cdot\nabla v_k$ on $Q_{0,T_k}$,
and conformal invariance gives
\[
\int_{Q_{0,T_k}}
\bigl(|\nabla v_k|^2+|\widetilde\Omega_k|^2\bigr)
\,dtd\theta
\le \Lambda.
\]

We first verify Definition~\ref{defi of almost hm} on subcylinders with small potential. Suppose that \(0<a<b<T_k\) and
\[
\|\widetilde\Omega_k\|_{L^2(Q_{a,b})}
\le \epsilon\le \epsilon_*,
\]
where $\epsilon_*>0$ is sufficiently small. For every three-block subcylinder \(Q_{s,s+3L}\subset Q_{a,b}\), the local estimate of Sharp–Topping \cite{SharpTopping2013} gives, for $0<r<r_0$,
\[
\begin{aligned}
\|\nabla^2v_k\|_{L^p(Q_{s+r,s+3L-r})}
&\le \tau(r)\|\nabla v_k\|_{L^2(Q_{s,s+3L})},\\
\|\Delta v_k\|_{L^1(Q_{s+r,s+3L-r})}
&\le \tau(r)\epsilon\|\nabla v_k\|_{L^2(Q_{s,s+3L})}.
\end{aligned}
\]

\textbf{Claim:} for every \(\eta>0\), there exist \(\epsilon_\eta\in(0,\epsilon_*]\) and \(H_\eta\ge3L\), independent of \(k,a,b\), such that
\begin{equation}\label{eq:riviere-uniform-core-estimate}
(E-E^*)(v_k,Q_{a+H_\eta,b-H_\eta})
\le \eta E(v_k,Q_{a,b}).
\end{equation}
whenever \(0<a<b<T_k\) and
\(
\|\widetilde\Omega_k\|_{L^2(Q_{a,b})}
\le\epsilon_\eta.
\)

Suppose that this fails for some fixed \(\eta>0\). Set
\(\epsilon_j=\min\{\epsilon_*,j^{-1}\}\). For every integer \(j\ge3\), there exist \(k_j\) and \(0<a_j<b_j<T_{k_j}\) such that
\(
\|\widetilde\Omega_{k_j}\|_{L^2(Q_{a_j,b_j})}
\le\epsilon_j
\)
and
\[
(E-E^*)(v_{k_j},Q_{a_j+jL,b_j-jL})
>\eta E(v_{k_j},Q_{a_j,b_j}).
\]
Necessarily, $\ell_j:=b_j-a_j>2jL$ and
$E(v_{k_j},Q_{a_j,b_j})>0$. Define
\[
z_j(t,\theta)
=
\frac{v_{k_j}(a_j+t,\theta)}
{\sqrt{E(v_{k_j},Q_{a_j,b_j})}}
\quad\text{on }Q_{0,\ell_j}.
\]
The preceding local estimates show that $z_j$ is
$(\tau,\epsilon_j)$-almost harmonic. Moreover,
\[
E(z_j,Q_{0,\ell_j})=1,
\qquad
(E-E^*)(z_j,Q_{jL,\ell_j-jL})>\eta.
\]
Set $m_j=\lfloor\ell_j/L\rfloor$.
Since $\epsilon_j\to0$ and the restrictions of $z_j$ to
$Q_{0,m_jL}$ have uniformly bounded energy,
Corollary \ref{intr.cor} gives
\[
\lim_{H\to\infty}\lim_{j\to\infty}
(E-E^*)(z_j,Q_{H,m_jL-H})=0.
\]
On the other hand, for every fixed $H>0$ and all sufficiently large $j$,
the inequalities $\ell_j-m_jL<L$ and $jL\ge H+L$ imply
\[
Q_{jL,\ell_j-jL}\subset Q_{H,m_jL-H},
\]
and hence
\[
(E-E^*)(z_j,Q_{H,m_jL-H})
\ge(E-E^*)(z_j,Q_{jL,\ell_j-jL})>\eta.
\]
This contradicts the preceding limit and proves the uniform estimate.

For \(T>0\), define
\[
h_k(T)
=
\sup_{T\leq s\leq T_k-T-L}
E(v_k,Q_{s,s+L}).
\]
The preliminary neck nonconcentration property, obtained from
local compactness and the energy gap, gives
\begin{equation}\label{eq:riviere-neck-nonconcentration}
\lim_{T\to\infty}\lim_{k\to\infty}h_k(T)=0,
\end{equation}
see \cite[Section~3.1, (61)--(62)]{LaurainRiviere2014}.

Fix \(\eta>0\) and \(T>0\). For all sufficiently large \(k\), we divide \(Q_{T,T_k-T}\) into \(J_k\) subcylinders \(Q_{a_{k,j},b_{k,j}}\) with disjoint interiors such that
\[
\int_{Q_{a_{k,j},b_{k,j}}}|\widetilde\Omega_k|^2\,dtd\theta
\le \epsilon_\eta^2,
\qquad
J_k\le 1+\frac{\Lambda}{\epsilon_\eta^2}.
\]
Hence, by \eqref{eq:riviere-uniform-core-estimate},
\[
\sum_{j=1}^{J_k}
(E-E^*)(v_k,Q_{a_{k,j}+H_\eta,b_{k,j}-H_\eta})
\le
\eta\sum_{j=1}^{J_k}E(v_k,Q_{a_{k,j},b_{k,j}})
\le \eta\Lambda.
\]
The remaining part of \(Q_{T,T_k-T}\) can be covered by at most \(C_{\eta,\Lambda}\) subcylinders \(Q_{s,s+L}\subset Q_{T,T_k-T}\). Combining the preceding estimate with \(0\le E-E^*\le E\), we obtain
\[
(E-E^*)(v_k,Q_{T,T_k-T})
\le \eta\Lambda+C_{\eta,\Lambda}h_k(T).
\]
Since $m=0$, we have
\[
\int_{Q_{T,T_k-T}}|\partial_\theta v_k|^2\,dt\,d\theta
\le (E-E^*)(v_k,Q_{T,T_k-T})
\le \eta\Lambda+C_{\eta,\Lambda}h_k(T).
\]
Using \eqref{eq:riviere-neck-nonconcentration} and taking first
$k\to\infty$, then $T\to\infty$, and finally
$\eta\to 0$, we obtain
\[
\lim_{T\to\infty}\lim_{k\to\infty}
\int_{Q_{T,T_k-T}}|\partial_\theta v_k|^2\,dtd\theta
=0.
\]
Taking $T=\log R$ and returning to annular coordinates, we have
\[
\int_{Q_{\log R,T_k-\log R}}
|\partial_\theta v_k|^2\,dtd\theta
=
\int_{D_{1/R}\setminus D_{Rr_k}}
|\nabla_\theta u_k|^2\,dx,
\]
where $\nabla_\theta u_k=\frac{1}{r}\partial_\theta u_k$. Consequently,
\begin{equation*}
\lim_{R\to\infty}\lim_{k\to\infty}
\int_{D_{1/R}\setminus D_{Rr_k}}
|\nabla_\theta u_k|^2\,dx
=0.
\end{equation*}
Thus no angular energy is lost along the neck, which proves the angular energy identity for Rivi\`ere's equation.

\subsection{Conformal immersion with small Willmore energy}Let $f_k : S^1 \times [0,m_kL] \to \mathbb{R}^n$ be a sequence of conformal immersions with $W(f_k) < \tau_0$, where $\tau_0$ is chosen such that
\[
\|\nabla f_k\|_{W^{1,p}(B_{r/2}(p))}
\le C \|\nabla f_k\|_{L^2(B_r(p))}
\]
(see \cite{ChenLi2014} for details). Here, conformality means
\[
|f_{k,t}|^2 = |f_{k,\theta}|^2.
\]

Assume that
\[
\Theta_k := \sup_i \|\nabla f_k\|_{L^2(Q_i)} \to 0.
\]
Then $f_k$ satisfies the assumptions of Theorem \ref{main1} with $m=0$.

Since $f_k$ is conformal, we have
\[
 \frac{\int_{Q_i}|f_{k,\theta}|^2}{\int_{Q_i}|\nabla f_k|^2}
= \frac{1}{2}.
\]
If \eqref{eqn:trivialenergy} does not occur, then by \eqref{eqn:a0E1} we must have
\[
 \frac{\int_{Q_i}|f_{k,\theta}|^2}{\int_{Q_i}|\nabla f_k|^2} \to 0
\qquad \text{for } \mathfrak a_k \le i \le \mathfrak b_k,
\]
which is a contradiction.

Therefore, \eqref{eqn:trivialenergy} must hold. In particular,
\[
E(f_k,Q_i)
\le C \Theta_k \big( e^{-iqL} + e^{-(m_k-i)qL} \big),
\]
which implies
\[
\lim_{k\to\infty} E(f_k) = 0,
\qquad
\lim_{k\to\infty} \operatorname{osc}_{S^1 \times [0,m_kL]} f_k = 0.
\]
This result was originally obtained in \cite{ChenLi2014} by an ODE argument. Here we provide an alternative proof based on the 3-circle analysis developed above.
\subsection{Harmonic mapping on long cylinders without weak conformality}
Let $u_k:(\Sigma,h_k)\rightarrow N$ be a sequence of harmonic maps, where $h_k$ is a fixed-area metric on $\Sigma$ with curvature $0$ or $-1$. When the corresponding complex structures of $h_k$ degenerate, the sequence $\{u_k\}$ may develop blow-up phenomena. 
In the degeneration process, certain collars of $\Sigma$ elongate and become long cylindrical regions. Analytically, this leads to the study of harmonic maps from \emph{long cylinders} into $N$ without weak weak conformality.

Now, we assume $u_k:S^1\times[0,m_kL]\rightarrow N$ is harmonic with
$$
\sup_i E(u_k,Q_i)\rightarrow 0.
$$
When the target manifold $N$ is isometrically embedded in $\mathbb{R}^n$, the harmonic map equation can be expressed using the \emph{second fundamental form} $A$ of $N$:
\[
\Delta u_k + A(u_k)(du_k, du_k) = 0.
\]
Parker \cite{Parker1996} first systematically studied such sequences and analyzed the associated blow-up phenomena. Chen and Tian \cite{ChenTian1999} later proved that if the sequence does not converge to a point, the image set in the blow-up limit converges to a geodesic in the target manifold $N$. Li and Wang \cite{LiWang2010} provided an alternative proof of this geodesic convergence. Theorem \ref{main2} gives a new proof.

Miaomiao Zhu \cite{Zhu2010} reformulated the Pohozaev identity as follows:
\[
\nu(u)=\int_{\{t\}\times S^1} \left(|u_t|^2 - |u_\theta|^2\right)\, d\theta = \text{constant in } t.
\] 
He use this constant to compute both the limiting energy and the limiting length of neck regions, i.e.
$$
\lim_{k\rightarrow+\infty}E(u_k)=2\pi\nu(u_k) m_kL,\s \lim_{k\rightarrow+\infty}\ell(u_k^*)=\sqrt{\nu(u_k)}m_kL,
$$
which is a direct corollary of Theorem \ref{main2}.

\subsection{$\alpha$-harmonic maps}

Let $(\Sigma,g)$ be a closed Riemann surface and let $N$ be a compact
Riemannian manifold isometrically embedded in $\mathbb R^K$. For
$\alpha>1$, the Sacks--Uhlenbeck functional is
\[
E_\alpha(u;\Omega)
=\int_\Omega (1+|\nabla_g u|^2)^\alpha\,dV_g.
\]
An $\alpha$-harmonic map satisfies
\begin{equation}\label{eq:alpha-map}
\Delta_g u+A(u)(\nabla_g u,\nabla_g u)
=-(\alpha-1)
\frac{\nabla_g|\nabla_g u|^2\cdot\nabla_g u}
     {1+|\nabla_g u|^2},
\end{equation}
where $A$ denotes the second fundamental form, with the sign
convention in \eqref{eq:alpha-map}. Moreover, we have the Pohozaev-type inequality:
\begin{eqnarray*}
 \left(1-\frac{1}{2\alpha}\right)\int_{\partial D_r(x)}(1+|\nabla_{g} u|^2)^{\alpha-1}
    \big|\frac{\partial u}{\partial r}\big|^2ds-\frac{1}{2\alpha}\int_{\partial
D_r(x)}(1
    +|\nabla_{g} u|^2)^{\alpha-1}\dis\big|u_{\theta}\big|^2ds\\
=\frac{(\alpha-1)}{\alpha r}\int_{D_r(x)}(1+|\nabla_{g}u|^2
    )^{\alpha-1}|\nabla_gu|^2dV_g
    +O(r).
\end{eqnarray*}
Choose an isothermal coordinate disk 
and write $g=e^{2\varphi}g_{\mathrm{euc}}$. Then
\begin{equation}\label{eq:alpha-map-euclidean}
\Delta u+A(u)(\nabla u,\nabla u)
=-(\alpha-1)
\frac{\nabla\!\left(e^{-2\varphi}|\nabla u|^2\right)\cdot\nabla u}
     {1+e^{-2\varphi}|\nabla u|^2}.
\end{equation}

We now consider a sequence of $\alpha_k$-harmonic maps
\[
u_k\colon (\Sigma,g)\to N,
\qquad
\alpha_k\to1,
\]
satisfying
\[
E_{\alpha_k}(u_k)\leq \Lambda.
\]

Assume, for simplicity, that $0$ is the only concentration point in a
fixed isothermal coordinate disk and there is only one bubble. For bubble tree cases,
we refer the reader to \cite{LiWang2010,ChenLi2014a} for details. 

Write
\[
g=e^{2\varphi}g_{\mathrm{euc}}
\]
in this disk. Then \eqref{eq:alpha-map} takes the form
\begin{equation}\label{eq:alpha-map-euclidean}
\Delta u_k+A(u_k)(\nabla u_k,\nabla u_k)
=
-(\alpha_k-1)
\frac{
\nabla\!\left(e^{-2\varphi}|\nabla u_k|^2\right)
\cdot\nabla u_k
}{
1+e^{-2\varphi}|\nabla u_k|^2
}.
\end{equation}

The neck analysis can be reduced to the study of the neck region $D_{\delta}\setminus D_{r_k/\delta}(x_k)$, where                     
$$
x_k\rightarrow 0,\s r_k\rightarrow 0,
$$
and $u_k(x_k+r_kx)$ converges to a bubble, which is a nontrivial harmonic map
from $S^2$ into $N$. 
Since there exists only one bubble, we have 
$$
\lim_{\delta\rightarrow 0}\lim_{k\rightarrow+\infty}\sup_{r\in[2r_k/\delta, \delta/2]}E(u_k,D_{2r}\setminus D_{r/2}(x_k))=0.
$$

Set $t=-\log r$, $x=x_k+(-\log r,\theta)$. 
Then \eqref{eq:alpha-map-euclidean} becomes
\begin{equation}\label{eq:alpha-map-euclidean2}
\Delta u_k+A(u_k)(\nabla u_k,\nabla u_k)
=-(\alpha_k-1)
\frac{\nabla\!\left(e^{-2\hat\varphi}|\nabla u_k|^2\right)\cdot\nabla u_k}
     {1+e^{-2\hat\varphi}|\nabla u_k|^2}
\end{equation}
on $[-\log\delta,-\log \frac{r_k}{\delta}]\times S^1$,
where
$$
\hat{\varphi}=\varphi\left(
x_k+e^{-t}(\cos\theta,\sin\theta)
\right)-t.
$$

By a result in
\cite{LiWang2010}, we have
$$
\|\nabla_gu_k\|_{C^0}^{\alpha_k-1}\leq C.
$$
Then the Pohozaev-type inequality implies that
\begin{eqnarray*}
\alpha_k-1&\leq& Cr\left(1-\frac{1}{2\alpha_k}\right)\int_{\partial D_r(x_k)}(1+|\nabla_{g} u_k|^2)^{\alpha_k-1}
    \big|\frac{\partial u_k}{\partial r}(x_k)\big|^2\,ds\\
&\leq& Cr^2\|\nabla u_k\|^2_{L^\infty(\partial D_r(x_k))}\\
&\leq& C\int_{D_{2r}\setminus D_{r/2}(x_k)}|\nabla u_k|^2\,dx\\
&=&C\int_{S^1\times(-\log r-\log 2,-\log r+\log 2)}|\nabla u_k|^2\,dtd\theta    
\end{eqnarray*}

Consequently, the right-hand side of
\eqref{eq:alpha-map-euclidean2} satisfies
\[
\left|
(\alpha_k-1)
\frac{
\nabla\!\left(
e^{-2\hat\varphi}|\nabla u_k|^2
\right)\cdot\nabla u_k
}{
1+e^{-2\hat\varphi}|\nabla u_k|^2
}
\right|
\leq
C(\alpha_k-1)
\left(
|\nabla^2u_k|+|\nabla u_k|
\right).
\]
By the $\epsilon$-regularity for $\alpha$-harmonic maps, $u_k$
is therefore an almost-harmonic perturbation as
$\alpha_k\rightarrow1$. We may consequently apply the
three-circle estimates established above.

First, we are able to obtain that
\begin{eqnarray*}
\lim_{\delta\rightarrow 0}\lim_{k\rightarrow+\infty}
\int_{D_{\delta}\setminus D_{r_k/\delta}(x_k)}|\nabla u_k|^2&\leq &
\lim_{\delta\rightarrow 0}\lim_{k\rightarrow+\infty}
\int_{D_{\delta}\setminus D_{r_k/\delta}(x_k)}(1+|\nabla_gu_k|^2)^{\alpha_k-1}|\frac{\partial u_k}{\partial r}|^2\\
&\leq&2\left(\lim_{k\rightarrow+\infty}\log r_k^{1-\alpha_k}\right)\lim_{\delta\rightarrow 0}
\int_{D_{\delta}(x_k)}(1+|\nabla_{g}u_k|^2
    )^{\alpha_k-1}|\nabla_gu_k|^2dV_g.
\end{eqnarray*}
This  implies a weak energy identity obtained in \cite{LiWang2010}. Moreover, $u_k$
has no energy loss on this neck region if and only if 
$$
\lim_{k\rightarrow+\infty}r_k^{1-\alpha_k}=1.
$$

In addition, the image of the neck region converges to a geodesic.
Let $\Gamma_k$ be the curve defined by
\[
u_k^*(r)
=
\frac{1}{2\pi r}
\int_{\partial D_r(x_k)}u_k\,ds.
\]
Although $\Gamma_k$ may not lie in $N$, since
$u_k-u_k^*\rightarrow 0$, the limit of $u_k^*$ must lie in $N$. By the Cauchy–Schwarz inequality and Lemma \ref{lem:est.a0}, together with the Pohozaev-type inequality and Corollary \ref{intr.cor}, we obtain
\begin{eqnarray*}
\lim_{\delta\rightarrow 0}\lim_{k\rightarrow+\infty}
\ell(\Gamma_k)
&=&
\lim_{\delta\rightarrow 0}\lim_{k\rightarrow+\infty}
\int_{r_k/\delta}^{\delta}|\dot u_k^*|\,dr\\
&\sim&
\lim_{\delta\rightarrow 0}\lim_{k\rightarrow+\infty}
\int_{r_k/\delta}^{\delta}
\sqrt{\int_{S^1}
\left|\frac{\partial u_k}{\partial r}\right|^2
(re^{i\theta})\,d\theta}\,dr\\
&\sim&
\lim_{\delta\rightarrow 0}\lim_{k\rightarrow+\infty}
\int_{r_k/\delta}^{\delta}
\sqrt{\int_{S^1}
(1+|\nabla u_k|^2)^{\alpha_k-1}
\left|\frac{\partial u_k}{\partial r}\right|^2
(re^{i\theta})\,d\theta}\,dr\\
&=&
\lim_{\delta\rightarrow 0}\lim_{k\rightarrow+\infty}
\int_{r_k/\delta}^{\delta}
\frac{
\sqrt{(\alpha_k-1)
\int_{D_{r}(x_k)}
(1+|\nabla_{g}u_k|^2)^{\alpha_k-1}
|\nabla_gu_k|^2dV_g}}{r}\,dr.
\end{eqnarray*}
Here, $x\sim y$ means that
\[
C^{-1}y\leq x\leq Cy
\]
for some constant $C>0$ independent of $k$.

Then, the limit has infinite length  if and only if
$$
\lim_{k\rightarrow+\infty}\sqrt{\alpha_k-1}\log\frac{1}{r_k}=+\infty.
$$
If it has finite length, then $-\sqrt{\alpha_k-1}\log r_k$
is bounded, which implies that $-(\alpha_k-1)\log r_k\rightarrow 0$, and
for any $s_k\in(2r_k/\delta,\delta/2)$, we have
\begin{eqnarray*}
 \|\nabla u_k\|_{L^\infty(\partial D_{s_k}(x_k))}^{2(\alpha_k-1)}&\leq&
\left(C\frac{E(u_k,D_{2s_k}\setminus D_{s_k/2}(x_k))}{s_k^2}\right)^{\alpha_k-1}\\
&\leq& \left(\frac{C}{r_k}\right)^{2\alpha_k-2}=1+o(1),\qquad k\to \infty,   
\end{eqnarray*}
which implies that
$$
(1+|\nabla u_k|^2)^{\alpha_k-1}(x_k+re^{i\theta})\rightarrow 1.
$$
Then
$$
\lim_{\delta\rightarrow 0}\lim_{k\rightarrow+\infty}\ell(\Gamma_k)=\left(\lim_{k\rightarrow+\infty}\log r_k^{-\sqrt{\alpha_k-1}}\right)
\lim_{\delta\rightarrow 0}\lim_{k\rightarrow+\infty}\sqrt{\int_{D_{\delta}(x_k)}(1+|\nabla_{g}u_k|^2
    )^{\alpha_k}\,dV_g}.
$$
Hence, we reprove results in \cite{LiWang2010}.\\

When the target manifold $(M,h)$ admits a Killing vector field $X \in \Gamma(TM)$, the energy functional is invariant under the one-parameter family of isometries generated by $X$. By Noether's theorem, this symmetry implies that the quantity
\[
\omega = (1+|\nabla u_k|^2)^{\alpha_k-1} \langle \nabla_g u_k, X(u_k) \rangle_h
\]
is divergence-free, i.e.
\[
\operatorname{div}_g \omega = 0.
\]
Then $\int_{\partial D_r}\omega=0$.

Assume that $u_k(c_k) \to p \in M$, and that $\gamma_k$ converges to a geodesic $\gamma$. Then
\[
\int_{S^1\times\{t_i\}}
\frac{\omega(u_k)}{\mathcal A_i(u_k)}
\;\longrightarrow\;
C\, \langle \dot\gamma, X(p) \rangle,
\]
where $C>0$ is a constant.

As a consequence, we obtain the following:
\begin{pro}
We have
\[
\dot\gamma(s)\in
\left(
\operatorname{span}\{X(\gamma(s)): X \text{ is a Killing vector field on } M\}
\right)^\perp .
\]
In particular, if the values of the Killing vector fields span
$T_xM$ at every point $x\in M$, then $\gamma$ is constant. Hence the
sequence has no nontrivial neck.
\end{pro}

This recovers a result proved in \cite{BayerRoberts2026}.

\begin{exa}
Let
\[
N=M\times T^m
\]
be equipped with the product metric
\[
h=h_M+h_{T^m},
\]
where $T^m$ is a flat torus. Let
\[
X_1,\ldots,X_m
\]
be the parallel Killing vector fields on $T^m$, viewed as Killing
vector fields on $N$. At every point of $N$, these vector fields span
the tangent space of the $T^m$-factor.
Suppose that the image of a neck region converges to a geodesic,
then must have
\[
\gamma\subset M\times\{q\}
\]
for some $q\in T^m$.
Thus the existence of Killing vector fields does not necessarily force
the limiting neck geodesic to be constant, but it may restrict the
geodesic to a lower-dimensional totally geodesic submanifold.
\end{exa}

\subsection{Energy identity of Willmore energy}
By the arguments in the appendix of \cite{LiYinZhou2024}, the study of bubble tree convergence for a sequence of closed Willmore surfaces with bounded $\|A\|_{L^2}$ reduces to the study of a sequence of conformal Willmore immersions
\[
f_k\colon[0,m_kL]\times S^1\to \mathbb{R}^n,
\]
with induced metric $g_k=f_k^*(g_{\mathbb{R}^n})=e^{2u_k}(dt^2+d\theta^2)$, satisfying:
\begin{itemize}
\item[A1)]  $u_k=-mt+v_k$ for some $m\in \mathbb Z_+$ independent of $k$, with
$$
\|v_k(0,\theta)\|_{L^{\infty}(S^1)}\le 1, \s\|\nabla v_k\|_{L^\infty([0,m_kL]\times S^1)}\rightarrow 0,
$$
\item[A2)]
$$\Theta_k=\sup_{t\in [0,m_kL-1]}\int_{[t,t+1]\times S^1}|A_k|^2\,{d}V_{g_k}\rightarrow 0,
$$
\item[A3)] $\int_0^{2\pi}f_k(m_kL-1,\theta)\,d\theta =0$.
\end{itemize}
We first show that the mean curvature $H$ of $f$ satisfies \eqref{almost.harmonic.1} and \eqref{almost.harmonic.2}.

Let $\nu_1$, $\cdots$, $\nu_{n-2}$ be an orthonormal basis of the normal bundle. Then
the Gauss map $\vec{n}$ can be written as
$$
\vec{n}=\nu_1\wedge \cdots\wedge \nu_{n-2}.
$$
By $\epsilon$-regularity (cf \cite{KuwertSchatzle2001},\cite{Riviere2008}), we have
$$
|\nabla \vec{n}|\rightarrow 0.
$$
We have the following Willmore Euler-Lagrange equation in divergence form of Rivi\`ere \cite{Riviere2008}
\begin{equation*}
\Delta {H}=\operatorname{div}\left( 3 \pi_{{n}}(\nabla {H}) - \star(\nabla^\bot {n} \wedge {H})\right),
\end{equation*}
or equivalently,
\begin{equation*}
\Delta {H} = \operatorname{div}\left(\frac{3}{2}\pi_T(\nabla {H}) + \star\frac{1}{2}(\nabla^\bot {n} \wedge {H})\right)
\end{equation*}
By the first equation, we obtain
$$
\Delta H=3\pi_n(\Delta H)+
3(\nabla\pi_n)(\nabla H)-\star(\nabla^\bot n\wedge\nabla H).
$$
We rewrite this as
$$
\Delta H=3\pi_n(\Delta H)+\hat\Omega_i H_i,
$$
which implies
$$
\pi_T(\Delta H)=\pi_T(\hat\Omega^iH_i).
$$
Therefore,
\begin{eqnarray*}    
\Delta H&=&\operatorname{div}\left(\frac{3}{2}\pi_T(\nabla H)+\star\frac{1}{2}(\nabla^\bot n\wedge H)\right)\\
&=&\frac{3}{2}\pi_T(\Delta H)+\frac{3}{2}(\nabla\pi_T)(\nabla H)+\star\frac{1}{2}(\nabla^\bot n\wedge\nabla H)\\
&=&\frac{3}{2}\pi_T(\hat\Omega^iH_i)+\frac{3}{2}(\nabla\pi_T)(\nabla H)+\star\frac{1}{2}(\nabla^\bot n\wedge\nabla H)\\
&:=&\Omega^iH_i.
\end{eqnarray*}

On $Q_i$, define 
$$
\hat f_k=e^{-u_k(t_i,0)}(f_k(t+t_i,\theta)-\vint_0^{2\pi}f_k(t_i,\theta)\,d\theta).
$$
Then, after choosing a suitable orthonormal basis $\{\delta_i\}$ of $\mathbb{R}^n$, we have smooth convergence
of $\hat f_k$ to
$$
\frac{1}{m}e^{-mt}(\cos m\theta,\sin m\theta,0,\cdots,0)=\frac{1}{m}e^{-mt}(\cos m\theta \delta_1+
\sin m\theta \delta_2)
$$
(see \cite{LiYin2025}). Set $\hat u_k(t,\theta)=u_k(t_i+t,\theta)-u_k(t_i,0)$.
Let 
$$
e_1=e^{-\hat u_k}\hat f_{k,1},\s e_2=e^{-\hat u_k}\hat f_{k,2}.
$$
Applying the Gram--Schmidt process, we construct an orthonormal basis $\{\hat\nu_\alpha\}$ of the normal bundle as follows:
\begin{eqnarray*}
 \hat{\nu}_1&=&\frac{\delta_3-\lan e_1,\delta_3\ran e_1-\lan e_2,\delta_3\ran e_2}{|\delta_3-\lan e_1,\delta_3\ran e_1-\lan e_2,\delta_3\ran e_2|},\\
 \hat{\nu}_2&=&\frac{\delta_4-\lan e_1,\delta_4\ran e_1-\lan e_2,\delta_4\ran e_2-\lan \hat{\nu}_1,\delta_4\ran \hat\nu_1 }{|\delta_4-\lan e_1,\delta_4\ran e_1-\lan e_2,\delta_4\ran e_2-\lan \hat{\nu}_1,\delta_4\ran\hat\nu_1|},\\
 \cdots&=&\cdots\\
 \hat{\nu}_{n-2}&=&\frac{\delta_n-\lan e_1,\delta_n\ran e_1-\lan e_2,\delta_n\ran e_2-\lan \hat{\nu}_1,\delta_n\ran \hat{\nu}_1 -\cdots-\lan\hat{\nu}_{n-3},\delta_n\ran\hat{\nu}_{n-3}}{|\delta_n-\lan e_1,\delta_n\ran e_1-\lan e_2,\delta_n\ran e_2-\lan \hat{\nu}_1,\delta_n\ran\hat{\nu}_1-\cdots-\lan\hat{\nu}_{n-3},\delta_n\ran\hat{\nu}_{n-3}|}.
\end{eqnarray*}
Then
$$
\|\nabla\hat{\nu}_\alpha\|_{C^0([-L/2,L/2]\times S^1)}\rightarrow 0,\quad \alpha=1,\cdots,n-2.
$$
Setting $\nu_\alpha(t,\theta)=\hat{\nu}_\alpha(t-t_i,\theta)$, we obtain an orthonormal basis of the normal bundle  on $Q_i$. We have
$$
\pi_n(\nabla H)=\sum_\alpha\lan\nabla H,\nu_\alpha\ran \nu_\alpha,
$$
and therefore
$$
\|\nabla\pi_n\|_{C^0(Q_i)}\leq C\sum_\alpha\|\nabla\nu_\alpha\|_{C^0(Q_i)}\rightarrow 0.
$$
Moreover, since
$$
\pi_T(\cdot)=\lan\cdot,f_i\ran g^{ij}f_j=Id-\pi_n(\cdot),
$$
we also have
$$
\|\nabla\pi_T\|_{C^0(Q_i)}\leq C\sum_\alpha\|\nabla \nu_\alpha\|_{C^0(Q_i)}\rightarrow 0.
$$
Consequently,
$$
\|\Omega^j(f_k)\|_{C^0(Q_i)}<\epsilon_k\rightarrow 0,\qquad j=1,2,\quad2\leq i\leq m_k-1.
$$

Let $\tau_1$ and $\tau_2$ denote the residues associated with the conservation laws corresponding to translations and rotations of $\mathbb{R}^n$, respectively.
By the arguments in \cite{LiYin2025}, if $\tau_1=0$ and $\tau_2=0$, then $\hat a_m=\hat b_m=0$. In this case, the 3-circle lemma applies and yields
\begin{equation}\label{eq:nabla Hk for 3 circle}
   \int_{Q_i}e^{-2mt}|\nabla H_k|^2\leq C\Theta_k(e^{-iqL}+e^{-(m_k-i)qL}). 
\end{equation}
Moreover, the trace embedding inequality, A1)–A2) and \eqref{eq:nabla Hk for 3 circle} give
\[
\int_{S^1}|H_k(L/2,\theta)|^2\,d\theta
\le C\int_{Q_1}\bigl(|H_k|^2+|\nabla H_k|^2\bigr)\,dtd\theta
\le C\Theta_k.
\]
By \eqref{eq:nabla Hk for 3 circle}, for \(L/2\le t\le(m_k-1)L\), we have
\[
\begin{aligned}
\|H_k(t,\cdot)\|_{L^2(S^1)}&\leq \|H_k(L/2,\cdot)\|_{L^2(S^1)}+\int_{L/2}^t\|\partial_\tau H_k(\tau,\cdot)\|_{L^2(S^1)}\,d\tau\\
&\le
\|H_k(L/2,\cdot)\|_{L^2(S^1)}
+
\sum_{i=1}^{\lceil t/L\rceil}
\left(\int_{(i-1)L}^{iL}e^{2m\tau}\,d\tau\right)^{1/2}
\left(\int_{Q_i}e^{-2m\tau}|\nabla H_k|^2\,d\tau d\theta\right)^{1/2}
\\
&\le
C\sqrt{\Theta_k}
\left[
1+\sum_{i=1}^{\lceil t/L\rceil}
\left(
e^{(m-q/2)iL}
+
e^{-qm_kL/2}e^{(m+q/2)iL}
\right)
\right]
\\
&\le
C\sqrt{\Theta_k}
\left(
1+e^{(m-q/2)t}
+e^{-qm_kL/2}e^{(m+q/2)t}
\right).
\end{aligned}
\]
Thus,
\[
\int_{S^1}|H_k(t,\theta)|^2\,d\theta
\le
C\Theta_k
\left(
1+e^{(2m-q)t}
+e^{(2m+q)t-qm_kL}
\right).
\]
Hence,
\[
\begin{aligned}
\int_{S^1\times[L,(m_k-1)L]}e^{-2mt}|H_k|^2\,dtd\theta
&\le
C\Theta_k\int_L^{(m_k-1)L}
\left(
e^{-2mt}+e^{-qt}+e^{-q(m_kL-t)}
\right)\,dt
\\
&\le C\Theta_k.
\end{aligned}
\]

When $\tau_2\neq 0$,  a similar argument gives
$$
\int_{Q_{0,\mathfrak a_k}}|H_k|^2+\int_{Q_{\mathfrak b_k,m_k}}|H_k|^2\leq C\Theta_k.
$$
Moreover, we have the following:

\begin{lem}
For any $c_k$ such that $c_k-\mathfrak a_k \to +\infty$ and $\mathfrak b_k-c_k \to +\infty$
$$
\frac{\int_{Q_{c_k}}e^{-2mt}|\nabla H_k|^2}{\int_{Q_{c_k}}e^{-2mt}|H_k|^2}\rightarrow 2m^2.
$$
\end{lem}
Then, it is easy to use $\tau_2$ to write down the limits of $\int_{Q_{0,m_kL}}|H_k|^2dV_{g_k}$
as in \cite{LiYinZhou2024}.

\section{Three-Circle Lemma for Inhomogeneous Perturbations}\label{sec:3-circle lemma to inho}
In this section, we study a more general class of perturbations of the harmonic map equation, where both nonlinear first-order terms and lower-order forcing terms are present.
More precisely, we consider $u\colon S^1\times[0,m_kL]\rightarrow\R^n$ which satisfies the following:
\begin{align*}
\|\nabla^2 u\|_{L^p(S^1\times[s+r,t-r])} 
&\leq \tau(r)\,(\|\nabla u\|_{L^2(S^1\times [s,t])}+\|f\|_{L^p(S^1\times[s,t])}), \\ 
\|\Delta u\|_{L^1(S^1\times[s+r,t-r])} 
&\leq \tau(r)\,(\epsilon\,\|\nabla u\|_{L^2(S^1\times[s,t])}+\|f\|_{L^1(S^1\times[s,t])}),
\end{align*}
where  $p>1$. 

We assume $\sup_iE(u_k,Q_i)\rightarrow 0$, and 
$$
\|f_k\|_{L^p(Q_i)}\leq e^{-iq_1L}A_k,\quad \text{where}~A_k\rightarrow 0.
$$

Let $1<i<m_k$. We say $Q_i$ is dominant if
$$
\frac{\|f_k\|_{L^p(Q_{i-1}\cup Q_i\cup Q_{i+1})}}{E^\frac{3}{2}(u_k,Q_i)}\geq 1,
$$
and  nondominant otherwise. It is easy to check that, if $Q_i$ is nondominant, then
$u_k$ satisfies
\eqref{almost.harmonic.1} and \eqref{almost.harmonic.2}
on $Q_{i-1}\cup Q_i\cup Q_{i+1}$.
\begin{lem}
Let
\[
I=\bigcup_{i=a}^b Q_i.
\]
Assume that $Q_a$ and $Q_b$ are dominant, while $Q_i$ is nondominant
for every $a<i<b$. Then there exists $q'>0$ such that
\[
E(u_k,Q_i)\leq C e^{-iq'L}A_k^{2/3}
\]
for every $a\leq i\leq b$.
\end{lem}

\begin{proof}
Since $Q_i$ is nondominant for every $a<i<b$, the forcing term is
controlled by
\[
\|f_k\|_{L^p(Q_{i-1}\cup Q_i\cup Q_{i+1})}
<
E(u_k,Q_i)^{3/2}.
\]
Hence, when the energy is sufficiently small, $u_k$ satisfies
\eqref{almost.harmonic.1} and \eqref{almost.harmonic.2} on the relevant
three-cylinder regions.

Then we can find $\mathfrak a$ and $\mathfrak b$ such that, for
$a\leq i\leq\mathfrak a$ and $\mathfrak b\leq i'\leq b$,
\[
E(u_k,Q_i)
\leq
C e^{-(i-a)qL}E(u_k,Q_a),
\]
and
\[
E(u_k,Q_{i'})
\leq
C e^{-(b-i')qL}E(u_k,Q_b),
\]
while
\[
\lambda_j\geq\beta_0
\qquad
\text{for every } \mathfrak a\leq j\leq\mathfrak b.
\]

Hence, Lemma~\ref{lem:ratio of lambda} can be applied successively to
the pairs $(Q_j,Q_{j+1})$, $j=i,\ldots,\mathfrak b-1$. After choosing
the almost-harmonic parameter sufficiently small, we obtain
\[
\mathcal A_j(u_k)
\leq
(1+\epsilon')\mathcal A_{j+1}(u_k).
\]
Iterating this inequality gives
\[
\mathcal A_i(u_k)
\leq
(1+\epsilon')^{\mathfrak b-i}
\mathcal A_{\mathfrak b}(u_k).
\]

Moreover, since $\lambda_i\geq\beta_0$, the definition of $\lambda_i$
gives
\[
E(u_k,Q_i)
\leq
\beta_0^{-2}\mathcal A_i^2(u_k),
\]
while Lemma~\ref{est.ahat} yields
\[
\mathcal A_{\mathfrak b}^2(u_k)
\leq
C E(u_k,Q_{\mathfrak b}).
\]
Therefore, for every $\mathfrak a\leq i\leq\mathfrak b$,
\[
\begin{aligned}
E(u_k,Q_i)
&\leq
C(1+\epsilon')^{2(\mathfrak b-i)}
E(u_k,Q_{\mathfrak b})\\
&\leq
C(1+\epsilon')^{2(\mathfrak b-i)}
e^{-(b-\mathfrak b)qL}E(u_k,Q_b).
\end{aligned}
\]

Since $Q_b$ is dominant,
\[
E(u_k,Q_b)^{3/2}
\leq
\|f_k\|_{L^p(Q_{b-1}\cup Q_b\cup Q_{b+1})}
\leq
C e^{-q_1bL}A_k,
\]
and hence
\[
E(u_k,Q_b)
\leq
C e^{-\frac{2q_1bL}{3}}A_k^{2/3}.
\]
Thus
\[
E(u_k,Q_i)
\leq
C(1+\epsilon')^{2(\mathfrak b-i)}
e^{-(b-\mathfrak b)qL}
e^{-\frac{2q_1bL}{3}}
A_k^{2/3}.
\]
Write
\[
(1+\epsilon')^{2(\mathfrak b-i)}
=
e^{\delta_{\epsilon'}(\mathfrak b-i)L},
\qquad
\delta_{\epsilon'}
=
\frac{2\log(1+\epsilon')}{L}.
\]
Since
\[
b=i+(\mathfrak b-i)+(b-\mathfrak b),
\]
we obtain
\[
\begin{aligned}
&\delta_{\epsilon'}(\mathfrak b-i)
-q(b-\mathfrak b)
-\frac{2q_1}{3}b\\
&\qquad
=
-\frac{2q_1}{3}i
-\left(\frac{2q_1}{3}-\delta_{\epsilon'}\right)
(\mathfrak b-i)
-\left(q+\frac{2q_1}{3}\right)(b-\mathfrak b).
\end{aligned}
\]
Choosing $\epsilon'>0$ so that
\(
\delta_{\epsilon'}<2q_1/3,
\)
we conclude that
\[
E(u_k,Q_i)
\leq
C e^{-\frac{2q_1}{3}iL}A_k^{2/3},\qquad\mathfrak a\leq i\leq\mathfrak b.
\]

On the other hand, since $Q_a$ is dominant,
\[
E(u_k,Q_a)
\leq
C e^{-\frac{2q_1aL}{3}}A_k^{2/3}.
\]
Hence, for $a\leq i\leq\mathfrak a$,
\[
\begin{aligned}
E(u_k,Q_i)
&\leq
C e^{-(i-a)qL}E(u_k,Q_a)\\
&\leq
C e^{-(i-a)qL}
e^{-\frac{2q_1aL}{3}}A_k^{2/3}\\
&\leq
C e^{-q'iL}A_k^{2/3},
\end{aligned}
\]
provided
\[
q'<\min\left\{q,\frac{2q_1}{3}\right\}.
\]

Similarly, for $\mathfrak b\leq i\leq b$,
\[
\begin{aligned}
E(u_k,Q_i)
&\leq
C e^{-(b-i)qL}E(u_k,Q_b)\\
&\leq
C e^{-(b-i)qL}
e^{-\frac{2q_1bL}{3}}A_k^{2/3}\\
&\leq
C e^{-q'iL}A_k^{2/3}.
\end{aligned}
\]
This completes the proof.
\end{proof}

Similarly, we have the following lemma.
\begin{lem}
Let
\[
I=\bigcup_{i=2}^b Q_i.
\]
Assume that $Q_b$ is dominant, while $Q_i$ is nondominant for every
$2\leq i<b$. Then there exists $q'>0$ such that
\[
E(u_k,Q_i)
\leq
C e^{-iq'L}
\left(A_k^{2/3}+E(u_k,Q_2)\right)
\]
for every $2\leq i\leq b$.
\end{lem}

We now distinguish two cases.

\medskip
\noindent
\textbf{Case 1.} Suppose that at least one block \(Q_i\), \(2\le i\le m_k-1\), is dominant, and let \(a_k\) be the last dominant index. Then \(Q_i\) is nondominant for every \(a_k<i\le m_k-1\). By the preceding lemmas,
\[
E(u_k,Q_i)
\le Ce^{-q'iL}\bigl(A_k^{2/3}+E(u_k,Q_2)\bigr),
\qquad 2\le i\le a_k,
\]
where \(q'>0\) and \(C\) are independent of \(k\) and \(i\). Since \(A_k\to0\) and \(\sup_iE(u_k,Q_i)\to0\), it follows that
\[
\sum_{i=1}^{a_k}E(u_k,Q_i)\to0.
\]

On the remaining cylinder \([a_kL,m_kL]\times S^1\), every interior three-block region is centered at a nondominant block. If \(m_k-a_k\) is bounded, the energy on this region tends to zero by the uniform smallness of the block energies. Otherwise, after passing to a further subsequence, we may assume that \(m_k-a_k\to\infty\). Theorem \ref{main1} then applies on this region after translating the cylinder and rescaling the maps to preserve the weighted energy. In particular, the corresponding energy estimates and normalized local limits hold on \([a_kL,m_kL]\times S^1\). If the total energy is uniformly bounded, Corollary \ref{intr.cor} gives
\[
\lim_{T\to\infty}\lim_{k\to\infty}(E-E^*)(u_k,[a_kL+T,m_kL-T]\times S^1)=0.
\]

Moreover, suppose that $m=0$ and
\[
\Delta u_k
=F_{0,k}(u_k,\nabla u_k,\nabla u_k)+G_k+f_k,
\]
where
\[
\int_{Q_i}|G_k|
\le C\epsilon_k
E(u_k,Q_{i-1}\cup Q_i\cup Q_{i+1}),
\qquad \epsilon_k\to0,
\]
and
\[
F_{0,k}(p,X,Y)=\alpha_k^{ij}(p)X_iY_j,
\]
with uniformly bounded continuous coefficients $\alpha_k^{ij}$
converging uniformly to $\alpha^{ij}$.
If the lengths of the averaged curves on $[a_kL,m_kL]$
do not tend to zero and the base points chosen as in Theorem \ref{main2}, then, after passing to a subsequence and
reparametrizing by arc length, these curves converge in
$C^1_{\mathrm{loc}}$ to a limit curve satisfying the equation
in Theorem~\ref{main2}, with
$F_0(p,X,Y)=\alpha^{ij}(p)X_iY_j$.
Indeed, on the remaining cylinder, the forcing term $f_k$
can be absorbed into the remainder $G_k$.

\medskip
\noindent
\textbf{Case 2.} Suppose that \(Q_i\) is nondominant for every \(2\le i\le m_k-1\). Then Theorem \ref{main1} applies directly on the whole cylinder \([0,m_kL]\times S^1\). As in Case 1, Corollary \ref{intr.cor} and Theorem \ref{main2} also apply on this cylinder under their respective additional assumptions.


\section{Approximation of the energy}\label{sec:approx}

We finally record a simple extension of the preceding results which
will be useful for variational approximations of the Dirichlet energy.
For simplicity, we restrict ourselves to the case $m=0$.

Let
\[
E(u,Q)=\int_Q |\nabla u|^2
\]
and suppose that, for each $k$, we are given an energy
\[
E_k(u,Q)=E(u,Q)+R_k(u,Q),
\]
where $R_k$ is nonnegative and additive on disjoint cylindrical
regions. We assume that
\begin{equation}\label{eq:approx.energy}
0\leq R_k(u,Q)
\leq
\epsilon_k\int_Q |\nabla^2u|^2,
\qquad
\epsilon_k\to0.
\end{equation}
We further assume that the analogues of
\eqref{almost.harmonic.1} and \eqref{almost.harmonic.2} hold with
$E$ replaced by $E_k$, and that
\begin{equation}\label{eq:approx.hessian}
\int_{Q_i}|\nabla^2u_k|^2
\leq
C E_k(u_k,Q_{i-1}\cup Q_i\cup Q_{i+1}),
\end{equation}
where $C$ is independent of $i$ and $k$.

We keep the definition of the critical energy unchanged. Namely, if
\[
u_k^*(t)
=
\frac{1}{2\pi}\int_{S^1}u_k(t,\theta)\,d\theta,
\]
we set
\[
E_k^*(u_k,Q_{a,b})
:=
E^*(u_k,Q_{a,b})
=
2\pi\int_a^b |\dot u_k^*(t)|^2\,dt.
\]
Thus
\[
0\leq E_k^*\leq E\leq E_k,
\]
and
\[
E_k-E_k^*
=
(E-E^*)+R_k.
\]
Accordingly, we define
\[
\lambda_{k,i}
=
\frac{2|\dot u_k^*(t_i)|}
     {E_k(u_k,Q_i)^{1/2}},
\qquad
\nu_{k,i}
=
\frac{(E_k-E_k^*)(u_k,Q_i)^{1/2}}
     {2|\dot u_k^*(t_i)|},
\]
where
\[
t_i=\left(i-\frac12\right)L.
\]

\begin{pro}\label{pro:approx.energy.3circle}
Under the above assumptions, all the three-circle estimates established
in the preceding sections remain valid with
\[
E,\qquad E-E^*,\qquad \lambda_i,\qquad \nu_i
\]
replaced respectively by
\[
E_k,\qquad E_k-E_k^*,\qquad
\lambda_{k,i},\qquad \nu_{k,i}.
\]
Consequently, Theorem~\ref{main1}, Corollary~\ref{intr.cor}, and
Theorem~\ref{main2} remain valid, with the corresponding occurrences
of $E$ replaced by $E_k$.
\end{pro}

\begin{proof}
We only need to verify that the additional term $R_k$ disappears in
the contradiction argument for the three-circle inequalities.

Indeed, suppose that a three-circle inequality fails and normalize by
\[
c_k^2=E_k(u_k,Q_2),
\qquad
v_k=\frac{u_k-\bar u_k}{c_k}.
\]
As in the proof of the original three-circle lemma, the failure of the
inequality implies
\[
E_k(u_k,Q_1\cup Q_2\cup Q_3)
\leq
C E_k(u_k,Q_2).
\]
The analogues of \eqref{almost.harmonic.1} and
\eqref{almost.harmonic.2} therefore imply, after passing to a
subsequence,
\[
v_k\rightharpoonup v
\quad\text{in }W^{2,p}_{\mathrm{loc}},
\qquad
v_k\to v
\quad\text{in }W^{1,2}_{\mathrm{loc}},
\]
where $v$ is harmonic.

On the other hand, by \eqref{eq:approx.energy} and
\eqref{eq:approx.hessian},
\[
\frac{R_k(u_k,Q_2)}{c_k^2}
\leq
C\epsilon_k
\frac{
E_k(u_k,Q_{1}\cup Q_2\cup Q_{3})
}{
E_k(u_k,Q_2)
}
\to 0.
\]
Hence
\[
\frac{E_k(u_k,Q_2)}{c_k^2}
\to
E(v,Q_2),
\]
and, since $E_k^*=E^*$,
\[
\frac{E_k^*(u_k,Q_2)}{c_k^2}
\to
E^*(v,Q_2).
\]
Therefore
\[
\frac{(E_k-E_k^*)(u_k,Q_2)}{c_k^2}
\to
(E-E^*)(v,Q_2).
\]

Thus every contradiction sequence converges to exactly the same
harmonic limit as in the proof of the original three-circle lemmas.
The proofs of the three-circle inequalities are therefore unchanged.
The long-cylinder decomposition and the proofs of
Theorem~\ref{main1}, Corollary~\ref{intr.cor}, and
Theorem~\ref{main2} depend only on these three-circle inequalities and
the corresponding almost-harmonic estimates. Hence the same arguments
apply with $E$ replaced by $E_k$.
\end{proof}

In particular, if
\[
E_k(u_k,S^1\times[0,m_kL])\leq\Lambda,
\]
then
\[
\lim_{T\to+\infty}\lim_{k\to\infty}
(E_k-E_k^*)
\bigl(
u_k,[T,m_kL-T]\times S^1
\bigr)
=0.
\]
Moreover, in the nontrivial central region,
\[
\lim_{r\to+\infty}\lim_{k\to\infty}
\max_{\mathfrak a_k+r<i<\mathfrak b_k-r}
\left(
\left|\frac{\pi L}{2}\lambda_{k,i}^2-1\right|
+\nu_{k,i}
+
\left|
\frac{E_k^*(u_k,Q_i)}
     {E_k(u_k,Q_i)}
-1
\right|
\right)
=0.
\]
After normalization at any sequence of cylinders staying away from the
two ends of the central region, the maps converge locally to a linear
function
\[
t\longmapsto At.
\]

We now apply Proposition~\ref{pro:approx.energy.3circle} to the
biharmonic approximation introduced by Lamm \cite{Lamm2006}. Let
$(\Sigma,g)$ be a closed Riemannian surface and let $N$ be a compact
Riemannian manifold isometrically embedded in $\mathbb R^\ell$. For
$\epsilon>0$, consider
\[
E_\epsilon(u)
=
\int_\Sigma
\left(
|\nabla_g u|^2
+
\epsilon |\Delta_g u|^2
\right)\,dV_g.
\]
Let $u_k$ be a sequence of critical points of $E_{\epsilon_k}$ with
\[
\epsilon_k\to0,
\qquad
E_{\epsilon_k}(u_k)\leq \Lambda.
\]

Assume that $0$ is a concentration point and that $r_k\to0$ is a
bubbling scale. Write \(g=e^{2\varphi}g_{\mathrm{euc}}\) in a conformal coordinate disk centered at \(0\). By Lamm’s blow-up analysis,
\[
\frac{\epsilon_k}{r_k^2}\to 0.
\]
We consider the neck region
\[
D_\delta\setminus D_{r_k/\delta}
\]
and introduce cylindrical coordinates \(t=-\log r\). Then
\[
D_\delta\setminus D_{r_k/\delta}
\cong
S^1\times
\left[
-\log\delta,
-\log\frac{r_k}{\delta}
\right].
\]
In these coordinates, the biharmonic part of the energy becomes
\[
\epsilon_k
\int
e^{2t-2\varphi}|\Delta u_k|^2\,dtd\theta.
\]
Thus, on a cylindrical region $Q$, we set
\[
E_k(u_k,Q)
=
\int_Q |\nabla u_k|^2\,dtd\theta
+
R_k(u_k,Q),
\]
where
\[
R_k(u_k,Q)
=
\epsilon_k
\int_Q e^{2t-2\varphi}|\Delta u_k|^2\,dtd\theta.
\]
Since \(t\le-\log(r_k/\delta)\) on the neck region, we have
\[
\epsilon_ke^{2t}
\le\delta^2\frac{\epsilon_k}{r_k^2}
=:\eta_k,
\qquad
\eta_k\to0.
\]
Since \(e^{-2\varphi}\) is bounded on the coordinate region,
\[
0\leq R_k(u_k,Q)
\leq
C\eta_k\int_Q |\Delta u_k|^2
\leq
C\eta_k\int_Q |\nabla^2u_k|^2.
\]
In particular, the biharmonic energy is of the form considered in
Proposition~\ref{pro:approx.energy.3circle}, with the small parameter
$\eta_k$ independent of the position of the cylinder $Q$ in the neck.

Moreover, the local small-energy estimates of Lamm imply that, on
every three consecutive cylinders,
\[
\int_{Q_i}|\nabla^2u_k|^2
\leq
C E_k(u_k,Q_{i-1}\cup Q_i\cup Q_{i+1}),
\]
and, for every fixed \(0<r<\frac{3L}{2}\),
\[
\|\nabla^2u_k\|_{L^p(S^1\times[(i-2)L+r,(i+1)L-r])}
\le C(r,p)
E_k(u_k,Q_{i-1}\cup Q_i\cup Q_{i+1})^{1/2}.
\]
Furthermore, using the Euler--Lagrange equation together with the same
local estimates, one obtains
\[
\|\Delta u_k\|_{L^1(S^1\times[(i-2)L+r,(i+1)L-r])}
\le\tau_k(r)
E_k(u_k,Q_{i-1}\cup Q_i\cup Q_{i+1})^{1/2}.
\]
where \(\tau_k(r)\to 0\) uniformly in $i$.

Hence all the assumptions of
Proposition~\ref{pro:approx.energy.3circle} are satisfied. Therefore
the three-circle estimates, the long-cylinder decomposition, and the
conclusions of Theorem~\ref{main1}, Corollary~\ref{intr.cor}, and
Theorem~\ref{main2} apply to the biharmonic approximation with $E$
replaced by $E_k$. Moreover, the neck consists geodesic if it exists.

\bibliographystyle{amsalpha}
\bibliography{references}

@article{JS09,
  author = {Jost, J\"urgen and Muazzez \c{S}im\c{s}ir, Fatma},
  title = {Affine harmonic maps},
  journal = {Analysis (Munich)},
  volume = {29},
  year = {2009},
  number = {2},
  pages = {185--197},
  issn = {0174-4747},
  mrnumber = {2554637},
  doi = {10.1524/anly.2009.1050},
  url = {https://doi.org/10.1524/anly.2009.1050},
}

@misc{LiYinZhou2024,
  author = {Li, Yuxiang and Yin, Hao and Zhou, Jie},
  title = {3-circle theorem for {W}illmore surfaces {II}---degeneration of the complex structure},
  year = {2024},
  note = {arXiv:2411.06453},
  doi = {10.48550/arXiv.2411.06453},
  url = {https://arxiv.org/abs/2411.06453},
}

@article{Zhu2010,
  author = {Zhu, Miaomiao},
  title = {Harmonic maps from degenerating {R}iemann surfaces},
  journal = {Math. Z.},
  fjournal = {Mathematische Zeitschrift},
  volume = {264},
  year = {2010},
  number = {1},
  pages = {63--85},
  issn = {0025-5874,1432-1823},
  mrclass = {58E20},
  mrnumber = {2564932},
  mrreviewer = {Andreas\ Gastel},
  doi = {10.1007/s00209-008-0452-0},
  url = {https://doi.org/10.1007/s00209-008-0452-0},
}

@article{WangWeiZhang2017,
  author = {Wang, Wendong and Wei, Dongyi and Zhang, Zhifei},
  title = {Energy identity for approximate harmonic maps from surfaces to general targets},
  journal = {J. Funct. Anal.},
  fjournal = {Journal of Functional Analysis},
  volume = {272},
  year = {2017},
  number = {2},
  pages = {776--803},
  issn = {0022-1236,1096-0783},
  mrclass = {58E20 (35J60 35R01 53C43)},
  mrnumber = {3571909},
  mrreviewer = {Deliang\ Xu},
  doi = {10.1016/j.jfa.2016.09.018},
  url = {https://doi.org/10.1016/j.jfa.2016.09.018},
}

@article{SacksUhlenbeck1981,
  author = {Sacks, J. and Uhlenbeck, K.},
  title = {The existence of minimal immersions of {$2$}-spheres},
  journal = {Ann. of Math. (2)},
  fjournal = {Annals of Mathematics. Second Series},
  volume = {113},
  year = {1981},
  number = {1},
  pages = {1--24},
  issn = {0003-486X},
  mrclass = {58E12 (53C42 58E20)},
  mrnumber = {604040},
  mrreviewer = {John\ C.\ Wood},
  doi = {10.2307/1971131},
  url = {https://doi.org/10.2307/1971131},
}

@article{SharpTopping2013,
  author = {Sharp, Ben and Topping, Peter},
  title = {Decay estimates for {R}ivi\`ere's equation, with applications to regularity and compactness},
  journal = {Trans. Amer. Math. Soc.},
  fjournal = {Transactions of the American Mathematical Society},
  volume = {365},
  year = {2013},
  number = {5},
  pages = {2317--2339},
  issn = {0002-9947,1088-6850},
  mrclass = {42B37 (35J47)},
  mrnumber = {3020100},
  mrreviewer = {Shaolei\ Ru},
  doi = {10.1090/S0002-9947-2012-05671-6},
  url = {https://doi.org/10.1090/S0002-9947-2012-05671-6},
}

@incollection{Schoen1984,
  author = {Schoen, Richard M.},
  title = {Analytic aspects of the harmonic map problem},
  booktitle = {Seminar on nonlinear partial differential equations ({B}erkeley, {C}alif., 1983)},
  series = {Math. Sci. Res. Inst. Publ.},
  volume = {2},
  pages = {321--358},
  publisher = {Springer, New York},
  year = {1984},
  isbn = {0-387-96079-1},
  mrclass = {58E20},
  mrnumber = {765241},
  mrreviewer = {Helmut\ Kaul},
  doi = {10.1007/978-1-4612-1110-5\_17},
  url = {https://doi.org/10.1007/978-1-4612-1110-5_17},
}

@article{Riviere2007,
  author = {Rivi\`ere, Tristan},
  title = {Conservation laws for conformally invariant variational problems},
  journal = {Invent. Math.},
  fjournal = {Inventiones Mathematicae},
  volume = {168},
  year = {2007},
  number = {1},
  pages = {1--22},
  issn = {0020-9910,1432-1297},
  mrclass = {58E20 (35J60)},
  mrnumber = {2285745},
  mrreviewer = {Andreas\ Gastel},
  doi = {10.1007/s00222-006-0023-0},
  url = {https://doi.org/10.1007/s00222-006-0023-0},
}

@article{Parker1996,
  author = {Parker, Thomas H.},
  title = {Bubble tree convergence for harmonic maps},
  journal = {J. Differential Geom.},
  fjournal = {Journal of Differential Geometry},
  volume = {44},
  year = {1996},
  number = {3},
  pages = {595--633},
  issn = {0022-040X,1945-743X},
  mrclass = {58E20},
  mrnumber = {1431008},
  mrreviewer = {Daniel\ Pollack},
  url = {http://projecteuclid.org/euclid.jdg/1214459224},
}

@article{Riviere2008,
  author = {Rivi\`ere, Tristan},
  title = {Analysis aspects of {W}illmore surfaces},
  journal = {Invent. Math.},
  fjournal = {Inventiones Mathematicae},
  volume = {174},
  year = {2008},
  number = {1},
  pages = {1--45},
  issn = {0020-9910,1432-1297},
  mrclass = {53C42 (35J60 53C45)},
  mrnumber = {2430975},
  mrreviewer = {Andreas\ Gastel},
  doi = {10.1007/s00222-008-0129-7},
  url = {https://doi.org/10.1007/s00222-008-0129-7},
}

@article{LaurainRiviere2014,
  author = {Laurain, Paul and Rivi\`ere, Tristan},
  title = {Angular energy quantization for linear elliptic systems with antisymmetric potentials and applications},
  journal = {Anal. PDE},
  fjournal = {Analysis \& PDE},
  volume = {7},
  year = {2014},
  number = {1},
  pages = {1--41},
  issn = {2157-5045,1948-206X},
  mrclass = {35J20 (32Q65 35B45 35B65 35J57 53C21 53D50 58E20)},
  mrnumber = {3219498},
  mrreviewer = {Siegfried\ Carl},
  doi = {10.2140/apde.2014.7.1},
  url = {https://doi.org/10.2140/apde.2014.7.1},
}

@article{Lamm2006,
  author = {Lamm, Tobias},
  title = {Fourth order approximation of harmonic maps from surfaces},
  journal = {Calc. Var. Partial Differential Equations},
  fjournal = {Calculus of Variations and Partial Differential Equations},
  volume = {27},
  year = {2006},
  number = {2},
  pages = {125--157},
  issn = {0944-2669,1432-0835},
  mrclass = {58E20 (35J60)},
  mrnumber = {2251990},
  mrreviewer = {T\^oru\ Nakajima},
  doi = {10.1007/s00526-005-0001-1},
  url = {https://doi.org/10.1007/s00526-005-0001-1},
}

@article{KuwertSchatzle2001,
  author = {Kuwert, Ernst and Sch\"atzle, Reiner},
  title = {The {W}illmore flow with small initial energy},
  journal = {J. Differential Geom.},
  fjournal = {Journal of Differential Geometry},
  volume = {57},
  year = {2001},
  number = {3},
  pages = {409--441},
  issn = {0022-040X,1945-743X},
  mrclass = {53C44},
  mrnumber = {1882663},
  mrreviewer = {Shu-Cheng\ Chang},
  url = {http://projecteuclid.org/euclid.jdg/1090348128},
}

@article{DingTian1995,
  author = {Ding, Weiyue and Tian, Gang},
  title = {Energy identity for a class of approximate harmonic maps from surfaces},
  journal = {Comm. Anal. Geom.},
  fjournal = {Communications in Analysis and Geometry},
  volume = {3},
  year = {1995},
  number = {3-4},
  pages = {543--554},
  issn = {1019-8385,1944-9992},
  mrclass = {58E20 (58G11)},
  mrnumber = {1371209},
  mrreviewer = {Joseph\ F.\ Grotowski},
  doi = {10.4310/CAG.1995.v3.n4.a1},
  url = {https://doi.org/10.4310/CAG.1995.v3.n4.a1},
}

@article{DingLiLiu2006,
  author = {Ding, Weiyue and Li, Jiayu and Liu, Qingyue},
  title = {Evolution of minimal torus in {R}iemannian manifolds},
  journal = {Invent. Math.},
  fjournal = {Inventiones Mathematicae},
  volume = {165},
  year = {2006},
  number = {2},
  pages = {225--242},
  issn = {0020-9910,1432-1297},
  mrclass = {53C21 (53C43 53C44 58E20)},
  mrnumber = {2231956},
  mrreviewer = {Weimin\ Sheng},
  doi = {10.1007/s00222-005-0491-7},
  url = {https://doi.org/10.1007/s00222-005-0491-7},
}

@article{QingTian1997,
  author = {Qing, Jie and Tian, Gang},
  title = {Bubbling of the heat flows for harmonic maps from surfaces},
  journal = {Comm. Pure Appl. Math.},
  fjournal = {Communications on Pure and Applied Mathematics},
  volume = {50},
  year = {1997},
  number = {4},
  pages = {295--310},
  issn = {0010-3640,1097-0312},
  mrclass = {58E20 (58G11)},
  mrnumber = {1438148},
  mrreviewer = {Daniel\ Pollack},
  doi = {10.1002/(SICI)1097-0312(199704)50:4<295::AID-CPA1>3.0.CO;2-5},
  url = {https://doi.org/10.1002/(SICI)1097-0312(199704)50:4<295::AID-CPA1>3.0.CO;2-5},
}

@article{ChenTian1999,
  author = {Chen, Jingyi and Tian, Gang},
  title = {Compactification of moduli space of harmonic mappings},
  journal = {Comment. Math. Helv.},
  fjournal = {Commentarii Mathematici Helvetici},
  volume = {74},
  year = {1999},
  number = {2},
  pages = {201--237},
  issn = {0010-2571,1420-8946},
  mrclass = {58E20 (58D27)},
  mrnumber = {1691947},
  mrreviewer = {Chiung-Jue\ Sung},
  doi = {10.1007/s000140050086},
  url = {https://doi.org/10.1007/s000140050086},
}

@article{BayerRoberts2026,
  author = {Bayer, Carolin and Roberts, Andrew M.},
  title = {Energy identity and no neck property for {$\varepsilon $}-harmonic and {$\alpha $}-harmonic maps into homogeneous target manifolds},
  journal = {Calc. Var. Partial Differential Equations},
  fjournal = {Calculus of Variations and Partial Differential Equations},
  volume = {65},
  year = {2026},
  number = {2},
  pages = {Paper No. 57, 18},
  issn = {0944-2669,1432-0835},
  mrclass = {58E20 (53C43)},
  mrnumber = {5015190},
  mrreviewer = {Xiangrong\ Zhu},
  doi = {10.1007/s00526-025-03233-w},
  url = {https://doi.org/10.1007/s00526-025-03233-w},
}

@article{LiWang2010,
  author = {Li, Yuxiang and Wang, Youde},
  title = {A weak energy identity and the length of necks for a sequence of {S}acks-{U}hlenbeck {$\alpha$}-harmonic maps},
  journal = {Adv. Math.},
  fjournal = {Advances in Mathematics},
  volume = {225},
  year = {2010},
  number = {3},
  pages = {1134--1184},
  issn = {0001-8708,1090-2082},
  mrclass = {58E20 (35J62)},
  mrnumber = {2673727},
  mrreviewer = {Andreas\ Gastel},
  doi = {10.1016/j.aim.2010.03.020},
  url = {https://doi.org/10.1016/j.aim.2010.03.020},
}

@article{ChenLi2014a,
  author = {Chen, Jingyi and Li, Yuxiang},
  title = {Homotopy classes of harmonic maps of the stratified 2-spheres and applications to geometric flows},
  journal = {Adv. Math.},
  fjournal = {Advances in Mathematics},
  volume = {263},
  year = {2014},
  pages = {357--388},
  issn = {0001-8708,1090-2082},
  mrclass = {53C44 (53C43 55Q05)},
  mrnumber = {3239142},
  mrreviewer = {Yongbing\ Zhang},
  doi = {10.1016/j.aim.2014.07.001},
  url = {https://doi.org/10.1016/j.aim.2014.07.001},
}

@article{ChenLi2014,
  author = {Chen, Jingyi and Li, Yuxiang},
  title = {Bubble tree of branched conformal immersions and applications to the {W}illmore functional},
  journal = {Amer. J. Math.},
  fjournal = {American Journal of Mathematics},
  volume = {136},
  year = {2014},
  number = {4},
  pages = {1107--1154},
  issn = {0002-9327,1080-6377},
  mrclass = {53C42 (53A10 53A30)},
  mrnumber = {3245189},
  mrreviewer = {Zejun\ Hu},
  doi = {10.1353/ajm.2014.0023},
  url = {https://doi.org/10.1353/ajm.2014.0023},
}

@article{LiYin2025,
  author = {Li, Yuxiang and Yin, Hao},
  title = {3-circle theorem for {W}illmore surface {I}},
  journal = {J. Lond. Math. Soc. (2)},
  fjournal = {Journal of the London Mathematical Society. Second Series},
  volume = {111},
  year = {2025},
  number = {5},
  pages = {Paper No. e70165, 28},
  issn = {0024-6107,1469-7750},
  mrclass = {53A05 (53C21)},
  mrnumber = {4899930},
  mrreviewer = {Dorian\ Martino},
  doi = {10.1112/jlms.70165},
  url = {https://doi.org/10.1112/jlms.70165},
}

@article{Branding2015,
  author = {Branding, Volker},
  title = {Magnetic {Dirac}-harmonic maps},
  journal = {Anal. Math. Phys.},
  fjournal = {Analysis and Mathematical Physics},
  volume = {5},
  year = {2015},
  number = {1},
  pages = {23--37},
  issn = {1664-2368,1664-235X},
  doi = {10.1007/s13324-014-0081-1},
  url = {https://doi.org/10.1007/s13324-014-0081-1},
}

\end{document}